\documentclass[final,twosided]{siamltex}
\usepackage{amsmath,amsfonts}
\usepackage{slashbox}
\usepackage{graphicx}
\usepackage{epstopdf}
\usepackage{mathrsfs}
\usepackage{mathrsfs,galois,booktabs,ctable,threeparttable,tabularx,algorithm,algorithmic}
\usepackage{subfigure}
\graphicspath{{./pics/}}

\newtheorem{remark}{Remark}[section]

\title{Alternating Direction Method of Multipliers \\
for Linear Inverse Problems }

\author{Yuling Jiao\thanks{The School of Statistics and Mathematics, Zhongnan University of Economics and Law, Wuhan 430063,
People's Republic of China} ({\tt yulingjiaomath@whu.edu.cn})
\and Qinian Jin\thanks{Mathematical Sciences Institute, Australian National
University, Canberra, ACT 2601, Australia} ({\tt qinian.jin@anu.edu.au})
\and Xiliang Lu\thanks{Corresponding author. School of Mathematics and Statistics, Wuhan University, Wuhan 430072, People's Republic of China
\&  Hubei Key Laboratory of Computational Science (Wuhan University), Wuhan 430072, People¡¯s Republic of China}
({\tt xllv.math@whu.edu.cn})
\and Weijie Wang\thanks{School of Mathematics and Statistics, Wuhan University, Wuhan 430072, People's Republic of China}
({\tt wjwang.math@whu.edu.cn})
}

\begin{document}

\maketitle

\def\p{\partial}
\def\l{\langle}
\def\r{\rangle}
\def\N{\mathscr N}
\def\R{\mathscr R}
\def\D{\mathscr D}
\def\H{\mathcal H}
\def\X{\mathcal X}
\def\Y{\mathcal Y}
\def\B{\mathcal B}
\def\A{\mathcal A}
\def\Z{\mathcal Z}
\def\a{\alpha}
\def\b{\beta}
\def\d{\delta}
\def\la{\lambda}
\def\X{\mathcal X}
\def\Y{\mathcal Y}

\begin{abstract}
In this paper
we propose an iterative method using alternating direction method of multipliers (ADMM) strategy to solve linear inverse problems in Hilbert
spaces with general convex penalty term. When the data is given exactly, we give a convergence
analysis of our ADMM algorithm without assuming the existence of Lagrange multiplier. In case the data contains noise,
we show that our method is a regularization method as long as it is terminated by a suitable stopping
rule. Various numerical simulations are performed to test the efficiency of the method.
\end{abstract}
\begin{keywords}
Linear inverse problems, alternating direction method of multipliers,
Bregman distance, convergence, regularization property.
\end{keywords}
\begin{AMS} 65J20, 65J22, 90C25
\end{AMS}
\pagestyle{myheadings}
\thispagestyle{plain}
\markboth{Y. Jiao, Q. Jin, X. Lu and W. Wang}{ADMM for Linear Inverse Problems}

\section{\bf Introduction}

The study of linear inverse problems gives rise to the linear  
equation of the form
\begin{equation}\label{eqn:gov}
Ax = b,
\end{equation}
where $A: \X \to \H$ is a bounded linear operator between two Hilbert spaces $\X$ and $\H$.
In applications, we will not simply consider (\ref{eqn:gov}) alone, instead
we will incorporate a priori available information on solutions into the problem.
Assume that we have a priori information on the feature, such as sparsity, of the sought solution under a suitable
transform $W$ from $\X$ to another Hilbert spaces $\Y$ with domain $\D(W)$. We may then take a convex
function $f: \Y \to (-\infty, \infty]$ to capture such feature. This leads us to consider the convex
minimization problem
\begin{align}\label{eq:1}
\left\{\begin{array}{lll}
\mbox{minimize } & f(Wx) \\
\mbox{subject to } & A x = b, \ \ \ x\in \D(W),
\end{array}\right.
\end{align}
where $f$ and $W$ should be specified during applications. For inverse problems, the operator $A$ usually is either
non-invertible or ill-conditioned with a huge condition number.
Thus, a small perturbation of the data may lead the problem
(\ref{eq:1}) to have no solution; even if it has a solution, this solution may not depend continuously on the data
due to the uncontrollable amplification of noise. In order to overcome such ill-posedness, regularization techniques
should be taken into account to produce a reasonable approximate solution from noisy data.
One may refer to \cite{EnglHankeNeubauer:1996,ItoJin:2014,SchusterBarbara:2012} for comprehensive accounts on
the variational regularization methods as well as iterative regularization methods.

Variational regularization methods typically consider (\ref{eq:1}) by solving a family of well-posed minimization problems
\begin{align}\label{eqn:Tik}
\min_{x\in \D(W)} \left\{\frac{1}{2} \|A x-b\|^2 + \a f(W x) \right\},
\end{align}
where $\a>0$ is the so called regularization parameter whose choice crucially affects the performance of the method.
When the regularization parameter $\a$ is given, many efficient solvers have been developed to solve
(\ref{eqn:Tik}) when $f$ are sparsity promoting functions. However,  to find a good approximate solution,
the regularization parameter $\a$ should be carefully chosen, consequently
one has to solve \eqref{eqn:Tik} for many different values of $\a$, which can be time-consuming.

Among algorithms for solving (\ref{eqn:Tik}), the alternating direction method of multipliers (ADMM) is a favorable one.
The ADMM was proposed in \cite{GabayMercier:1976,GlowinskiMarroco:1975} around the mid-1970 and was analyzed in \cite{EcksteinBertsekas:1992,Gabay:1983,LionsMercier:1979,Rockafellar:1976}. It has been widely used
in solving structured optimization problems due to its decomposability and superior flexibility.
Recently, it has been revisited and popularized in modern signal/image processing, statistics, machine learning,
and so on; see the recent survey paper \cite{Boydetal:2011} and the references therein.
Due to the popularity of ADMM and its variants, new and  refined  convergence results  have been
obtained from several different perspectives, see, for example,
\cite{ChambollePock:2011,DavisYin:2014b,DengYin:2012,HeLiuWangYuan:2014,HeYuan:2012b,HeYuan:2012a,HongLuo:2012,LiShenXuZhang:2014,ZhangBurgerOsher:2011},
just to name a few of them. To the best of our knowledge, the existing convergence analyses
of ADMM depend on the solvability of the dual problem or the existence of saddle points for the
corresponding Lagrangian function, which might not be true
for inverse problems \eqref{eq:1}.

In this paper we propose an ADMM algorithm in the framework of iterative regularization methods.
By introducing an additional variable $y = W x$, we can reformulate (\ref{eq:1}) into the
equivalent form
 \begin{align}\label{eq:2}
\left\{\begin{array}{lll}
\mbox{minimize } &  f(y) \\
\mbox{subject to } &  A x = b, \ \ \ Wx =y, \,\,\, x\in \D(W).
\end{array}\right.
\end{align}
The corresponding augmented Lagrangian function is
\begin{align}\label{eq:lag}
L_{\rho_1,\rho_2}(x, y; \la, \mu) & = f(y) + \l \la, A x-b\r + \l \mu, Wx-y\r\nonumber \\
& \quad \,  + \frac{\rho_1}{2} \|A x-b\|^2 + \frac{\rho_2}{2}\|Wx-y\|^2,
\end{align}
where $\rho_1$ and $\rho_2$ are two positive constants. Our ADMM algorithm then reads
\begin{equation}\label{ADMM}
\left\{\begin{array}{l}
x_{k+1} = \arg \min_{x\in \D(W)} \;L_{\rho_1,\rho_2} (x,y_k;\lambda_k,\mu_k),\\[1.2ex]
y_{k+1} = \arg \min_{y\in \Y} \;L_{\rho_1,\rho_2} (x_{k+1},y;\lambda_k,\mu_k),\\[1.2ex]
\lambda_{k+1}= \lambda_k + \rho_1 (A x_{k+1} -b), \\[1.2ex]
\mu_{k+1} = \mu_k + \rho_2(Wx_{k+1}-y_{k+1}).
\end{array}
\right.
\end{equation}
The $x$-subproblem in \eqref{ADMM} is a quadratical minimization problem, which can be solved by many methods,
and the $y$-subproblem can be solved explicitly when $f$ is properly chosen, e.g., $f$ are certain sparsity
promoting functions, Thus, our ADMM algorithm can be efficiently implemented.
When the data $b$ in (\ref{eq:1}) is consistent in the sense that $b = A x$ for some $x\in \D(W)$ with $W x\in \D(f)$,
and when $f$ is strongly convex, we give a convergence analysis of our ADMM algorithm by using tools
from convex analysis. The proof is based on several remarkable monotonicity results and does not need the existence of the Lagrange multiplier to \eqref{eq:2}.
When the data contains noise, similar as other iterative regularization methods, our ADMM algorithm shows
the semi-convergence property, i.e. the iterate becomes close to the sought solution at the beginning,
however, after a critical number of iterations, the iterate leaves the sought solution far away as the iteration
proceeds. By proposing a suitable stopping rule, we establish the regularization property of our algorithm.
To the best of our knowledge, this is the first time that ADMM is used to solve inverse problems directly as
an iterative regularization method.

There are several different iterative regularization methods proposed for solving (\ref{eq:1}).
The augmented Lagrangian method (ALM) is a popular and efficient algorithm which takes the form
\begin{equation}\label{eqn:ALM}
\left\{\begin{array}{l}
x_{k+1} = \arg \min_{x\in \D(W)} \left\{f(Wx) +\l \la_k, Ax - b\rangle + \frac{\rho_k}{2}\|Ax - b\|^2\right\},\\[1.4ex]
\la_{k+1} = \la_k + \rho_k (Ax_{k+1} -b),
\end{array}
\right.
\end{equation}
where $\{\rho_k\}$ is a sequence of positive numbers satisfying suitable properties.
ALM was originally proposed by Hestenes \cite{Hestenes:1969} and Powell \cite{Powell:1969} independently;
see \cite{Bertsekas:1982,Rockafellar:1973} for its convergence analysis for well-posed optimization problems.
Recently, ALM has been applied to solve ill-posed inverse problems in  \cite{FrickGrasmair:2012,FrickLorenzElena:2011,FrickScherzer:2010,JinZhong:2014,OBGXY:2005}. It has been shown
that ALM can produce a satisfactory approximate solution within a few iterations if $\{\rho_k\}$ is chosen to
be geometrically increasing. However, since $f(Wx)$ and $A$ are coupled, solving the $x$-subproblem in (\ref{eqn:ALM})
is highly nontrivial, and an additional inner solver should be incorporated into the algorithm.
To remedy this drawback, a linearization step can be introduced to modify the $x$-subproblem in (\ref{eqn:ALM})
which leads to the Uzawa-type iteration
\begin{equation}\label{eqn:UZAWA}
\left\{ \begin{array}{l}
\la_{k+1} = \la_k + \rho_k (Ax_{k} -b),\\[1.0ex]
x_{k+1} = \arg \min_{x\in \D(W)}\left\{f(Wx) + \langle \la_{k+1}, Ax - b\rangle\right\}.
\end{array}
\right.
\end{equation}
This method and its variants have been analyzed in \cite{BotHein:2012,JinLu:2014,JinWang:2013}
for ill-posed inverse problems. Unlike (\ref{eqn:ALM}), the resolution of the $x$-subproblem in (\ref{eqn:UZAWA})
only relies on $f(W x)$ and hence is much easier for implementation. For instance, if the sought solution
is sparse, one may take $f(x) = \|x\|_{\ell^1} + \frac{\nu}{2}\|x\|^2$ and $W = I$ the identity,
then the $x$-subproblem in \eqref{eqn:UZAWA} can be solved explicitly by the soft thresholding.
However, if the sought solution is sparse under a transform $W$ that is not identity, which can occur when using
the total variation \cite{RudinOsherFatemi:1992}, the wavelet frame \cite{DongShen:2010,Shen:2010}, and so on,
the $x$-subproblem in (\ref{eqn:UZAWA}) does not have a closed form solution and an inner solver is needed.
In contrast to (\ref{eqn:ALM}) and (\ref{eqn:UZAWA}), our ADMM algorithm (\ref{ADMM}) can admit closed form solutions
for each subproblem for many important applications.

The rest of this paper is organized as follows. In Section 2, we give conditions to guarantee
that (\ref{eq:1}) has a unique solution, show that our ADMM (\ref{ADMM}) is well-defined and
establish an important monotonicity result. When the data is given exactly, we provide various convergence results
of (\ref{ADMM}). When the data contains noise, we propose a stopping rule
to terminate the iteration and show that our ADMM renders into a regularization method.
In Section 3 we report various numerical results to test the efficiency of our ADMM algorithm.
Finally, we draw conclusions in Section 4.

\section{\bf The method and its convergence analysis}

\subsection{Preliminary}
We consider the convex minimization problem (\ref{eq:1}) arising from linear inverse problems,
where $A: \X \to \H$ is a bounded linear operator between two Hilbert spaces $\X$ and $\H$,
$W$ is a linear operator from $\X$ to another Hilbert spaces $\Y$ with domain $\D(W)$, and
$f: \Y \to (-\infty, \infty]$ is a convex function. The inner products and norms on $\mathcal{X}$,
$\mathcal{Y}$ and $\mathcal{H}$ will be simply denoted by $\l \cdot, \cdot\r$ and $\|\cdot\|$ respectively,
which should be clear from the context.
Throughout the paper we will make the following assumptions on the operators $W$, $A$ and the function $f$:
\begin{itemize}
\item[(\textbf{A1})] $A: \X \to \H$ is a bounded linear operator. We use
$A^*: \H\to \X$ to denote its adjoint.

\item[(\textbf{A2})] $f: \mathcal{Y}\to (-\infty, \infty]$ is a proper, lower semi-continuous, strongly convex function
in the sense that there is a constant $c_0>0$ such that
\begin{align}\label{eq:sc}
f(t y_1+ (1-t) y_2) + c_0 t(1-t) \|y_1-y_2\|^2 \le t f(y_1) + (1-t) f(y_2)
\end{align}
for all $y_1, y_2 \in \Y$ and $0\le t\le 1$.

\item[(\textbf{A3})] $W: \X \to \Y$ is a densely defined, closed, linear operator with domain $\D(W)$.

\item[(\textbf{A4})] There is a constant $c_1>0$
such that
$$
\|A x\|^2 + \|W x\|^2 \ge c_1 \|x\|^2, \qquad \forall x\in \D(W).
$$
\end{itemize}

The assumptions (\textbf{A1}) and (\textbf{A2}) are standard.
We will use $\p f(y)$ to denote the subdifferential of $f$ at $y$, i.e.
$$
\p f(y) = \{ \mu\in \Y: f(\bar y) \ge f(y) + \l \mu, \bar y - y\r \mbox{ for all } \bar y \in \Y\}.
$$
Let $\D(\p f) = \{ y \in \mathcal{Y}: \p f(y) \ne \emptyset\}$. Then for $y\in \D(\p f)$ and $\mu \in \p f(y)$
we can introduce
$$
D_{\mu} f(\bar y, y) = f(\bar y) -f(y) -\l \mu, \bar y-y\r, \quad \forall \bar y\in \mathcal{Y}
$$
which is called the Bregman distance induced by $f$ at $y$ in the direction $\mu$, see \cite{Bregman:1967}.
When $f$ is strongly convex in the sense of (\ref{eq:sc}), by definition one can show that
\begin{align}\label{eq:BD}
D_\mu f(\bar y, y) \ge c_0 \|\bar y-y\|^2
\end{align}
for all $\bar y\in \Y$, $y\in \D(\p f)$ and $\mu \in \p f(y)$. Moreover
\begin{align}\label{eq:77}
\l \mu-\bar \mu, y-\bar y\r \ge 2 c_0 \|y-\bar y\|^2
\end{align}
for all $y, \bar y \in \D(\p f)$, $\mu \in \p f(y)$ and $\bar \mu \in \p f(\bar y)$.

The assumptions (\textbf{A3}) and (\textbf{A4}) are standard conditions used in the literature on
regularization methods with differential operators, see \cite{EnglHankeNeubauer:1996,LP:1980,Morozov:1984},
they will be used to show that our ADMM (\ref{ADMM}) is well-defined. The closedness
of $W$ in (\textbf{A3}) implies that $W$ is also weakly closed. We can define the adjoint $W^*$
of $W$ which is also closed and densely defined. Moreover, $z \in \D(W^*)$ if and only if
$
\l W^* z, x\r = \l z, W x\r
$
for all $x\in \D(W)$. A sufficient condition which guarantees (\textbf{A4}) is that
\begin{align}\label{eq:suff}
W \mbox{ has a closed range in } \Y, \quad  \dim \textrm{Ker}\; W < +\infty, \quad
\textrm{Ker}\; W \cap \textrm{Ker}\; A = \{0\};
\end{align}
see \cite[Chapter 8]{EnglHankeNeubauer:1996}. This sufficient condition is important in practice
as it is satisfied by many interesting examples. For instance, consider the following three examples:
\begin{itemize}
\item[(i)] $W = I$ the identity operator with $\D(W) = \X =\Y$.

\item[(ii)] $W$ is a frame transform, i.e., there exist $0< c\leq C<+\infty$ such that
$$
c\|\varphi\|^{2}\leq\|W\varphi\|^{2}\leq C \|\varphi\|^{2}, \quad \forall \varphi \in \X
$$
with $\D(W) = \X = L^2(\Omega)$ and $\Y = \ell^2(\mathbb{N})$.

\item[(iii)] The constant function $\mbox{1}$ is not in the kernel of $A$, $W = \nabla$ the gradient operator with $\X = L^2(\Omega)$, $\D(W) = H^1(\Omega)$ and $\Y = [L^2(\Omega)]^d$.
\end{itemize}
For (i), the conditions (\textbf{A3}) and (\ref{eq:suff}) hold trivially. For (ii), (\textbf{A3}) follows from
\cite[Proposition 12.7]{conway:1990} and (\ref{eq:suff}) follows from the coercivity of $W$ (see e.g. \cite[pp. 107]{conway:1990}).
For (iii), when $\Omega\subset {\mathbb R}^d$ with $d\le 3$ is an open bounded domain with Lipschitz boundary,
(\textbf{A3}) follows from the definition of weak derivatives. Note that $\textrm{Ker}\; W$ is a one dimension
subspace. This fact together with the Helmholtz-Hodge decomposition (see \cite[Theorem 3.4]{GiraultRaviart:1979})
implies (\ref{eq:suff}).

Under the assumptions (\textbf{A1})--(\textbf{A4}), the following result shows that the minimization problem (\ref{eq:1})
admits a unique solution whenever $b$ is consistent in the sense that $b= Ax$ for some $x\in \D(W)$ with $W x \in \D(f)$.

\begin{theorem}\label{thm:exist}
Let $b$ in (\ref{eq:1}) be consistent and let \emph{(\textbf{A1})--(\textbf{A4})} hold. Then the optimization problem \eqref{eq:1}
admits a unique solution $x^* \in \D(W)$ with $W x^* \in \D(f)$.
\end{theorem}

\begin{proof}
Let $f_*: = \inf \{f(Wz): Az = b, \, z\in \D(W)\}$. Since $b$ is consistent, we have $f_*<\infty$.
Let $\{z_k\}$ be the minimizing sequence such that
\begin{equation*}
z_k \in \D(W), \quad A z_k = b \quad \mbox{and} \quad \lim_{k\rightarrow \infty} f(W z_k) = f_*.
\end{equation*}
By (\textbf{A2}), $f$ is strongly convex and hence is coercive,  see e.g. \cite[Proposition 11.16]{BauschkeCombettes:2011}.
Thus $\{W z_k\} $ is bounded in $\Y$. In view of (\textbf{A4}),
$\{z_k\}$ is bounded in $\X$. Therefore, $\{z_k\}$ has a subsequence, which is denoted
by the same notation, such that
$$
z_k \rightarrow x^* \textrm{  weakly in  }\X,\quad W z_k \rightarrow y^* \textrm{  weakly in  }\Y.
$$
By using (\textbf{A3}) and $\{z_k\}\subset \D(W)$, we have $x^*\in \D(W)$ and $y^* = W x^*$.
Since $f$ is convex and lower semi-continuous, $f$ is also weakly lower semi-continuous
(see \cite[Proposition 10.23]{BauschkeCombettes:2011}). Thus
$f(W x^*) \le \liminf_{k\rightarrow \infty} f(W z_k) = f_*$ and hence $x^*$ is an optimal
solution of (\ref{eq:1}). The uniqueness follows by the strong convexity of $f$ and (\textbf{A4}).
\end{proof}

\subsection{ADMM algorithm and basic estimates}

As described in the introduction, by introducing an additional variable $y=W x$, we can reformulate
\eqref{eq:1} into the equivalent form (\ref{eq:2}). Recall the augmented Lagrangian function \eqref{eq:lag},
our ADMM (\ref{ADMM}) starts from some initial guess $y_0\in \Y$, $\la_0\in \H$, $\mu_0 \in \Y$ and defines
\begin{align}
x_{k+1} & = \arg\min_{x\in \D(W)} \left\{\l \la_k, A x\r + \l \mu_k, Wx\r + \frac{\rho_1}{2} \|A x-b\|^2
   + \frac{\rho_2}{2}\|Wx-y_k\|^2\right\}, \label{eq:3}\\
y_{k+1} & = \arg\min_{y\in \Y} \left\{f(y) - \l \mu_k, y\r + \frac{\rho_2}{2}\|Wx_{k+1}-y\|^2\right\}, \label{eq:4}\\
\la_{k+1} & = \la_k + \rho_1 (A x_{k+1} -b), \label{eq:5}\\
\mu_{k+1} & = \mu_k + \rho_2(Wx_{k+1}-y_{k+1}) \label{eq:6}
\end{align}
for $k =0,1, \cdots$, where $\rho_1$ and $\rho_2$ are two fixed positive constants.

We need to show that $x_{k+1}$ and $y_{k+1}$ are well-defined. Note that (\ref{eq:3}) and (\ref{eq:4})
can be written as
\begin{align*}
x_{k+1} & = \arg\min_{x\in \D(W)} \left\{\frac{\rho_1}{2} \|A x-b+\la_k/\rho_1\|^2
   + \frac{\rho_2}{2}\|Wx-y_k+\mu_k/\rho_2\|^2\right\},\\
y_{k+1} & = \arg\min_{y\in \Y} \left\{f(y) + \frac{\rho_2}{2}\|y-Wx_{k+1}-\mu_k/\rho_2\|^2\right\}.
\end{align*}
Therefore, the well-posedness of $x_{k+1}$ and $y_{k+1}$ follows from the  following result.

\begin{lemma}\label{lem:w2}
Let Assumptions \emph{(\textbf{A1})--(\textbf{A4})} hold.

\begin{itemize}
\item[(i)] For any $h\in \H$ and $v\in \Y$, the minimization problem
\begin{align}
\min_{z\in\D(W)} \left\{\frac{\rho_1}{2} \|A z - h\|^2 + \frac{\rho_2}{2} \|Wz - v\|^2\right\} \label{eq:tmx}
\end{align}
admits a unique solution $z$. Moreover, $z$ and $W z$ depend continuously on $h$ and $v$.

\item[(ii)] For any $v\in \Y$ the minimization problem
\begin{align}
\min_{y \in\Y} \left\{f(y) + \frac{\rho_2}{2} \|y - v\|^2\right\} \label{eq:tmy}
\end{align}
admits a unique solution $y$. Moreover, $y$ and $f(y)$ depend continuously on $v$.
\end{itemize}
\end{lemma}

\begin{proof}
(i) follows from \cite[pp. 23 Theorem 4 and pp. 26 Theorem 6]{Morozov:1993} and (ii) follows from \cite[Lemma 2.2]{JY2015}.
\end{proof}

We now take a closer look at the ADMM algorithm (\ref{eq:3})--(\ref{eq:6}).
From (\ref{eq:3}) it follows that $x_{k+1} \in \D(W)$ satisfies the optimality condition
$$
\l A^*\la_k + \rho_1 A^*(Ax_{k+1}-b), x\r + \l \mu_k+ \rho_2(W x_{k+1}-y_k), W x\r =0, \quad x\in \D(W).
$$
This implies that $\mu_k +\rho_2(W x_{k+1}-y_k) \in \D(W^*)$ and
\begin{align}\label{eq:7}
A^* \la_k + \rho_1A^*(A x_{k+1}-b)+ W^*[\mu_k+\rho_2(Wx_{k+1}-y_k)]=0.
\end{align}
From (\ref{eq:4}) we can obtain that
\begin{align}\label{eq:8}
0 \in \p f(y_{k+1}) -\mu_k -\rho_2(Wx_{k+1}-y_{k+1}).
\end{align}
For simplicity of exposition, we introduce the residuals
$$
r_k = A x_k -b \quad \mbox{and} \quad  s_k = Wx_k-y_k, \quad k =1,2, \cdots.
$$
It then follows from (\ref{eq:5}), (\ref{eq:6}), (\ref{eq:7}) and (\ref{eq:8}) that
\begin{align}
\la_{k+1} -\la_k & = \rho_1 r_{k+1}, \label{eq:9}\\
\mu_{k+1} -\mu_k & = \rho_2 s_{k+1}, \label{eq:10}\\
\mu_{k+1} & \in \p f(y_{k+1}), \label{eq:11}\\
A^* \la_k +\rho_1A^* r_{k+1}& = -W^* [\mu_k+\rho_2 (W x_{k+1}-y_k)] \label{eq:12}
\end{align}
for $k=0,1, \cdots$.

\begin{lemma}\label{lem1}
There holds $A^*\la_1= W^*[\rho_2 (y_0-y_1)-\mu_1]$. Moreover, for $k\ge 1$ there holds
$
\rho_1A^* r_{k+1}  = \rho_2W^*[(y_k-y_{k+1}) -(y_{k-1}-y_k)-s_{k+1}],
$
that is
$$
\rho_1\l r_{k+1}, A x\r  = \rho_2\l (y_k-y_{k+1}) -(y_{k-1}-y_k)-s_{k+1}, W x\r
$$
for all $x \in \D(W)$.
\end{lemma}

\begin{proof}
From (\ref{eq:9}) we have
$
A^* \la_{k+1}-A^* \la_k = \rho_1 A^* r_{k+1}
$
for $k\ge 0$. This together with (\ref{eq:10}) and (\ref{eq:12}) gives for $k\ge 0$ that
\begin{align}\label{eq:13}
A^*\la_{k+1} &= A^* \la_k  +\rho_1 A^* r_{k+1} + W^*(\rho_2 s_{k+1}+\mu_k-\mu_{k+1}) \nonumber\\
&= W^*[\rho_2(y_{k}-y_{k+1})-\mu_{k+1}]
\end{align}
which in particular implies $A^*\la_1 =W^*[\rho_2(y_0-y_1)-\mu_1]$. Further, by using (\ref{eq:9}),
(\ref{eq:13}) and (\ref{eq:10}) we have
\begin{align*}
\rho_1A^* r_{k+1} &=A^* \la_{k+1}-A^* \la_k =W^*\left[\rho_2 (y_k-y_{k+1})-\rho_2(y_{k-1}-y_k) +\mu_k-\mu_{k+1}\right]\\
& = \rho_2W^*[(y_k-y_{k+1}) -(y_{k-1}-y_k)-s_{k+1}]
\end{align*}
for $k\ge 1$, which completes the proof.
\end{proof}

The following monotonicity result plays an essential role in the forthcoming convergence analysis.

\begin{lemma}\label{lem3}
Let $E_k =\rho_1 \|r_{k}\|^2 +\rho_2\|s_k\|^2 +\rho_2 \|y_k-y_{k-1}\|^2$. Then
$$
E_{k+1} - E_k \le - \rho_1\|r_{k+1}-r_k\|^2 - 4c_0 \|y_{k+1}-y_k\|^2
$$
for all $k\ge 1$. In particular, $E_k$ is monotonically decreasing along the iteration
and $\sum_{k=1}^\infty \|y_{k+1}-y_{k}\|^2<\infty$.
\end{lemma}

\begin{proof}
By the definition of $r_{k}$ and $s_k$ we have
\begin{align}
r_{k+1} - r_k & = A(x_{k+1}-x_k), \label{eq:14}\\
s_{k+1} - s_k & = W(x_{k+1}-x_k)+(y_k-y_{k+1}). \label{eq:15}
\end{align}
Therefore
\begin{align*}
& \rho_1\left\l r_{k+1}-r_k, r_{k+1}\right\r + \rho_2 \left\l s_{k+1}-s_k, s_{k+1}\right\r \\
& = \rho_1\left\l A (x_{k+1}-x_k), r_{k+1}\right\r + \rho_2 \left\l W(x_{k+1}-x_k)+(y_k-y_{k+1}), s_{k+1}\right\r.
\end{align*}
Recall that $x_{k+1}-x_k \in \D(W)$. We may use Lemma \ref{lem1}, (\ref{eq:10}), (\ref{eq:11}) and (\ref{eq:77}) to derive that
\begin{align*}
& \rho_1\l r_{k+1}-r_k, r_{k+1}\r + \rho_2\l s_{k+1}-s_k, s_{k+1}\r \\
& = \rho_2\l W(x_{k+1}-x_k), (y_k-y_{k+1})-(y_{k-1}-y_k)\r -\l y_{k+1}-y_k, \mu_{k+1}-\mu_k\r\\
& \le \rho_2\l W(x_{k+1}-x_k), (y_k-y_{k+1})-(y_{k-1}-y_k)\r -2 c_0\| y_{k+1}-y_k\|^2.
\end{align*}
In view of (\ref{eq:15}) and the Cauchy-Schwarz inequality, we have
\begin{align*}
& \rho_1\l r_{k+1}-r_k, r_{k+1}\r + \rho_2\l s_{k+1}-s_k, s_{k+1}\r \\
& \le \rho_2\l (s_{k+1}-s_k) +(y_{k+1}-y_k), (y_k-y_{k+1})-(y_{k-1}-y_k)\r -2 c_0\| y_{k+1}-y_k\|^2,\\
& \le \frac{\rho_2}{2} \left(\|s_{k+1}-s_k\|^2 - \|y_{k+1}-y_k\|^2 +\|y_k-y_{k-1}\|^2\right)-2c_0\| y_{k+1}-y_k\|^2.
\end{align*}
By virtue of the identity $2\l a-b, a\r= \|a\|^2 -\|b\|^2 + \|a-b\|^2$, we therefore obtain
\begin{align*}
\rho_1\|r_{k+1}\|^2 +\rho_2\|s_{k+1}\|^2
& \le \rho_1\|r_k\|^2 +\rho_2\|s_k\|^2 -\rho_1\|r_{k+1}-r_k\|^2 -\rho_2\|y_{k+1}-y_k\|^2 \\
& \quad \, +\rho_2\|y_k-y_{k-1}\|^2 -4 c_0\| y_{k+1}-y_k\|^2.
\end{align*}
This shows the desired inequality.
\end{proof}

\subsection{Exact data case}

In this subsection we will give the convergence analysis of the ADMM algorithm (\ref{eq:3})--(\ref{eq:6})
under the condition that the data $b$ is consistent so that (\ref{eq:1})
has a unique solution. We will always use $(\hat{x}, \hat{y})$ to represent any feasible point of \eqref{eq:2},
i.e., $\hat x \in \D(W)$ and $\hat{y}\in \D(f)$ such that $A \hat x =b$ and $W\hat{x}=\hat{y}$.

\begin{lemma}\label{lem4}
The sequences $\{x_k\}$ and $\{y_k\}$ are bounded and
$$
\sum_{k=1}^\infty \left\{ D_{\mu_k} f(y_{k+1}, y_k) + E_k \right\} <\infty.
$$
In particular, $A x_k \rightarrow b$, $Wx_k-y_k\rightarrow 0$ and $y_{k+1}-y_k\rightarrow 0$
as $k\rightarrow \infty$.
\end{lemma}

\begin{proof}
Let $(\hat{x},\hat{y})$ be any feasible point of (\ref{eq:2}). By using (\ref{eq:10}) and Lemma \ref{lem1} we have
\begin{align}\label{eq:28}
&D_{\mu_{k+1}}  f(\hat y, y_{k+1}) - D_{\mu_k} f(\hat y, y_k) +  D_{\mu_k} f(y_{k+1}, y_k) \nonumber\\
& = \l \mu_k-\mu_{k+1}, \hat y - y_{k+1}\r = - \rho_2 \l s_{k+1}, W(\hat x - x_{k+1})+ s_{k+1}\r  \nonumber\\
& = - \rho_2 \|s_{k+1}\|^2 +\rho_2\l (y_{k-1}-y_k) -(y_k-y_{k+1}), W(\hat x - x_{k+1})\r  \nonumber \\
& \quad \, +\rho_1\l r_{k+1}, A(\hat x-x_{k+1})\r \nonumber\\
& = -\rho_1 \|r_{k+1}\|^2 -\rho_2\|s_{k+1}\|^2 + \rho_2 \l y_{k-1}-y_k, W(\hat x-x_{k+1})\r \nonumber \\
& \quad \, -\rho_2\l y_k-y_{k+1}, W(\hat x - x_{k+1})\r.
\end{align}
For any positive integers $m<n$, by summing the above inequality over $k$ from $k=m$ to $k=n-1$ we can obtain
\begin{align} \label{eq:16}
& D_{\mu_n}f(\hat y, y_n) -D_{\mu_m} f(\hat y, y_m) +\sum_{k=m}^{n-1} D_{\mu_k} f(y_{k+1}, y_k) \displaybreak[0] \nonumber\\
& = - \sum_{k=m+1}^{n} (\rho_1\|r_k\|^2 +\rho_2\|s_k\|^2) + \rho_2 \l y_{m-1}-y_{m}, W(\hat x-x_{m+1})\r \displaybreak[0] \nonumber\\
& \quad \, + \rho_2 \sum_{k=m}^{n-2} \l y_k-y_{k+1}, W(x_{k+1} -x_{k+2})\r -\rho_2 \l y_{n-1}-y_{n}, W(\hat x-x_{n})\r.\displaybreak[0]
\end{align}
By taking $m=1$ in the above equation, it follows that
\begin{align*}
& D_{\mu_n}f(\hat y, y_n) +\sum_{k=1}^{n-1} D_{\mu_k} f(y_{k+1}, y_k)\\
& = D_{\mu_1} f(\hat y, y_1) -  \sum_{k=2}^{n} (\rho_1\|r_k\|^2 +\rho_2\|s_k\|^2) + \rho_2 \l y_{0}-y_{1}, W(\hat x-x_{2})\r \displaybreak[0]\\
& \quad \, + \rho_2 \sum_{k=1}^{n-2} \l y_k-y_{k+1}, W(x_{k+1}-x_{k+2})\r -\rho_2 \l y_{n-1}-y_{n}, W(\hat x-x_{n})\r.\displaybreak[0]
\end{align*}
We need to estimate the last two terms. By the Cauchy-Schwarz inequality, we have
\begin{align}\label{eq:29}
\sum_{k=1}^{n-2} & \l y_k-y_{k+1}, W(x_{k+1}-x_{k+2})\r
  = \sum_{k=1}^{n-2} \l y_k -y_{k+1}, s_{k+1} +(y_{k+1}-y_{k+2})-s_{k+2}\r \displaybreak[0] \nonumber\\
& \le \sum_{k=1}^{n-2} \left(\frac{1}{8} \|s_{k+1}\|^2 +\frac{1}{8} \|s_{k+2}\|^2 + \frac{9}{2}\|y_k-y_{k+1}\|^2
 + \frac{1}{2} \|y_{k+1}-y_{k+2}\|^2\right)   \nonumber\\
&\le \frac{1}{4} \sum_{k=2}^{n} \|s_k\|^2 + 5 \sum_{k=1}^{n-1} \|y_k-y_{k+1}\|^2.
\end{align}
Similarly we have
\begin{align}\label{eq:299}
& -\l y_{n-1} -y_{n},  W(\hat x-x_{n})\r \nonumber \\
& =-\l y_{n-1}-y_{n}, W(\hat x-x_1)\r -\sum_{k=1}^{n-1}\l y_{n-1}-y_{n}, s_k + (y_k-y_{k+1})-s_{k+1}\r \displaybreak[0] \nonumber\\
& \le  \frac{1}{4} \|W(\hat x-x_1)\|^2  - \l y_{n-1}-y_{n}, s_1-s_n\r
-\sum_{k=1}^{n-2} \l y_{n-1}-y_{n}, y_k-y_{k+1}\r \displaybreak[0] \nonumber\\
& \le \frac{1}{4} \|W(\hat x-x_1)\|^2 + \frac{1}{4} \left(\|s_1\|^2 +\|s_n\|^2\right) + \frac{n}{4} \|y_{n-1}-y_{n}\|^2
+ 2\sum_{k=1}^{n-2} \|y_k-y_{k+1}\|^2.
\end{align}
Therefore, we can conclude that there is a constant $C$ independent of $n$ such that
\begin{align}\label{eq:30}
& D_{\mu_n}f(\hat y, y_n) + \sum_{k=1}^{n-1} D_{\mu_k} f(y_{k+1}, y_k) \nonumber\\
& \le C -\rho_1 \sum_{k=2}^{n} \|r_k\|^2 - \frac{\rho_2}{2} \sum_{k=2}^{n} \|s_k\|^2 + \frac{1}{4}\rho_2 n \|y_{n-1}-y_{n}\|^2
 + 7 \rho_2 \sum_{k=1}^{n-1} \| y_k-y_{k+1}\|^2.
\end{align}
From Lemma \ref{lem3} it follows that
$
\sum_{n=1}^\infty \|y_n-y_{n+1}\|^2 <\infty.
$
Thus, we can find a subsequence of integers $\{n_j\}$ with $n_j\rightarrow \infty$ such that
$
n_j \|y_{n_j}-y_{n_j+1}\|^2 \rightarrow 0
$
as $j\rightarrow \infty$. Consequently, it follows from (\ref{eq:30}) that
\begin{align*}
& \sum_{k=1}^{n_j-1} D_{\mu_k} f(y_{k+1}, y_k)
+\rho_1 \sum_{k=2}^{n_j} \|r_k\|^2 +\frac{\rho_2}{2} \sum_{k=2}^{n_j} \|s_k\|^2 \le C.
\end{align*}
Letting $j\rightarrow \infty$ gives
$$
\sum_{k=1}^\infty \left( D_{\mu_k} f(y_{k+1}, y_k) + \rho_1\|r_k\|^2 +\rho_2\|s_k\|^2\right) <\infty.
$$
We therefore obtain $\sum_{k=1}^\infty E_k<\infty$. By Lemma \ref{lem3}, $\{E_k\}$ is
monotonically decreasing. Thus $n E_n \le \sum_{k=1}^n E_k \le C$, and $n \rho_2\|y_n-y_{n+1}\|^2 \le n E_n \le C$.
Consequently, from (\ref{eq:30}) it follows that
$D_{\mu_n} f(\hat y, y_n)\le C$. By the strong convexity of $f$, we can conclude that $\{y_n\}$
is bounded. Furthermore, using $\sum_{n=1}^\infty E_n<\infty$,
we can conclude that $Ax_n\rightarrow b$, $Wx_n-y_n\rightarrow 0$ and
$y_n-y_{n+1}\rightarrow 0$ as $n\rightarrow \infty$. In view of the boundedness of $\{A x_n\}$ and $\{W x_n\}$,
we can use (\textbf{A4}) to conclude that $\{x_n\}$ is bounded.
\end{proof}

\begin{lemma}\label{lem5}
Let $(\hat x, \hat y)$ be any feasible point of (\ref{eq:2}). Then $\{D_{\mu_k} f(\hat y, y_k)\}$
is a convergent sequence.
\end{lemma}

\begin{proof}
Let $m<n$ be any two positive integers. By using (\ref{eq:16}) we have
\begin{align*}
\left| D_{\mu_n} f(\hat y, y_n) - D_{\mu_m} f(\hat y, y_m)\right|
& \le  \sum_{k=m}^{n-1} D_{\mu_k} f(y_{k+1}, y_k) + \sum_{k=m+1}^{n} \left(\rho_1\|r_k\|^2 +\rho_2\|s_k\|^2\right) \displaybreak[0]\\
& \quad \, +\rho_2 |\l y_{m-1}-y_m, W(\hat x-x_{m+1})\r| \\
& \quad \, +\rho_2 \left|\sum_{k=m}^{n-2} \l y_k-y_{k+1}, W(x_{k+1}-x_{k+2})\r \right| \\
& \quad \, + \rho_2 |\l y_{n-1}-y_{n}, W(\hat x-x_{n})\r |.
\end{align*}
By the same argument for deriving (\ref{eq:29}), we have
$$
\left|\sum_{k=m}^{n-2} \l y_k-y_{k+1}, W(x_{k+1}-x_{k+2})\r \right|
\le \frac{1}{4} \sum_{k=m+1}^n \|s_k\|^2 + 5 \sum_{k=m}^{n-1} \|y_k-y_{k+1}\|^2.
$$
Therefore
\begin{align*}
& \left| D_{\mu_n} f(\hat y, y_n) - D_{\mu_m} f(\hat y, y_m)\right| \\
& \le \sum_{k=m}^\infty\left( D_{\mu_k} f(y_{k+1}, y_k) + \rho_1\|r_k\|^2 +\frac{5}{4}\rho_2\|s_k\|^2 + 5 \rho_2\|y_k-y_{k+1}\|^2\right)\\
& \quad \, +\rho_2 \| y_{m-1}-y_m\| \|W(\hat x-x_{m+1})\| + \rho_2 \|y_{n-1}-y_{n}\| \|W(\hat x-x_{n})\|.
\end{align*}
In view of Lemma \ref{lem4}, we can conclude that
\begin{align*}
\left| D_{\mu_n} f(\hat y, y_n) - D_{\mu_m} f(\hat y, y_m)\right| \rightarrow 0 \quad \mbox{ as }
m, n \rightarrow \infty.
\end{align*}
This shows that $\{D_{\mu_k} f(\hat y, y_k)\}$ is a Cauchy sequence and hence is convergent.
\end{proof}

Now we are ready to give the main convergence result concerning the ADMM algorithm (\ref{eq:3})--(\ref{eq:6}) with
exact data.

\begin{theorem}\label{thm:exact}
Let \emph{(\textbf{A1})--(\textbf{A4})} hold and let $b$ be consistent. Let $x^*$ be the unique
solution of (\ref{eq:1}) and let $y^* = Wx^*$. Then for the ADMM (\ref{eq:3})--(\ref{eq:6}) there hold
$$
x_k\rightarrow x^*, \quad y_k \rightarrow y^*, \quad W x_k \rightarrow y^*,\quad f(y_k) \rightarrow f(y^*)\quad
\mbox{and} \quad D_{\mu_k} f(y^*, y_k)\rightarrow 0
$$
as $k\rightarrow \infty$.
\end{theorem}

\begin{proof}
We first show that $\{y_k\}$ is a Cauchy sequence. Let $(\hat{x}, \hat{y})$ be any feasible point of (\ref{eq:2}).
We will use the identity
\begin{align}\label{eq:17}
D_{\mu_m} f(y_n, y_m) = D_{\mu_m} f(\hat y, y_m) - D_{\mu_n} f(\hat y, y_n) + \l \mu_n-\mu_m, y_n-\hat y\r.
\end{align}
In view of (\ref{eq:10}) and lemma (\ref{lem1}), we can write
\begin{align*}
\l \mu_n-\mu_m, y_n-\hat y\r
& =\sum_{k=m}^{n-1} \l \mu_{k+1}-\mu_k, y_n-\hat y\r
= \rho_2 \sum_{k=m}^{n-1} \l s_{k+1}, y_n-\hat y\r \\
& = -\rho_2 \sum_{k=m}^{n-1} \l s_{k+1}, s_n \r +\rho_2 \sum_{k=m}^{n-1} \l s_{k+1}, W(x_n-\hat x)\r\\
& = - \sum_{k=m}^{n-1} \left(\rho_2\l s_{k+1}, s_n\r +\rho_1 \l r_{k+1}, r_n\r\right)
+ \rho_2 \l y_{n-1}-y_n, W(x_n-\hat x)\r \nonumber \\
& \quad \,  -\rho_2 \l y_{m-1}-y_m,W( x_n-\hat x)\r.
\end{align*}
By the Cauchy-Schwarz inequality and the monotonicity of $\{E_k\}$, we have
\begin{align}\label{eq:18}
|\l \mu_n-\mu_m, y_n-\hat y\r|
& \le \frac{1}{2}\sum_{k=m+1}^{n} \left(\rho_1\|r_k\|^2+ \rho_2\|s_k\|^2\right) + \frac{n-m}{2}  \left(\rho_1\|r_n\|^2 +\rho_2\|s_n\|^2\right) \nonumber\\
& \quad + \rho_2| \l y_{n-1}-y_n, W(x_n-\hat x)\r| + \rho_2 |\l y_{m-1}-y_m, W(x_n-\hat x)\r| \nonumber\\
& \le \sum_{k=m+1}^{n} E_k + \rho_2 |\l y_{n-1}-y_n,W( x_n-\hat x)\r| \nonumber \\
&\quad \,  + \rho_2 |\l y_{m-1}-y_m, W(x_n-\hat x)\r|.
\end{align}
This together with Lemma \ref{lem4} implies that
\begin{align}\label{eq:33}
\l \mu_n-\mu_m, y_n-\hat y\r \rightarrow 0\quad \mbox{ as } m,n\rightarrow \infty.
\end{align}
Combining this with (\ref{eq:17}) and using Lemma \ref{lem5},
we obtain that $D_{\mu_m} f(y_n, y_m)\rightarrow 0$ as $m, n\rightarrow \infty$. By the strong convexity of $f$ we have
$\|y_n-y_m\|\rightarrow 0$ as $m,n\rightarrow \infty$. Thus $\{y_k\}$ is a Cauchy sequence in $\Y$.
Consequently, there is $\tilde{y}\in \Y $ such that $y_k\rightarrow \tilde{y}$ as $k\rightarrow \infty$.

We will show that there is $\tilde{x}\in \D(W)$ such that $x_k\rightarrow \tilde{x}$, $A\tilde{x} = b$
and $W\tilde{x} = \tilde{y}$. By virtue of Lemma \ref{lem4} and $y_k \rightarrow \tilde{y}$, we have
$Wx_k \rightarrow \tilde{y}$ and $A x_k \rightarrow b$. From (\textbf{A4}) it follows that
$$
c_1 \|x_n-x_k\|^2 \le \|A x_n-A x_k\|^2 + \|W x_n -W x_k\|^2
$$
for any integers $n,k$. Therefore $\|x_n-x_k\| \rightarrow 0$ as $n, k\rightarrow \infty$ which shows that
$\{x_k\}$ is a Cauchy sequence in $\X$. Thus, there is $\tilde x\in \X$ such that $x_k \rightarrow \tilde x$
as $k\rightarrow \infty$. Clearly $b = \lim_{k\rightarrow \infty} A x_k = A \tilde x$. By using the closedness
of $W$ and $\{x_k\}\subset \D(W)$, we can further conclude that $\tilde x \in \D(W)$ and $W \tilde x = \tilde y$.

Next we will show that
$$
\tilde{y} \in \D(f), \quad \lim_{k\rightarrow \infty} f(y_k)=f(\tilde{y}) \quad \mbox{ and } \quad
\lim_{k\rightarrow \infty} D_{\mu_k} f(\tilde{y}, y_k) = 0.
$$
Recall that $\mu_k\in \p f(y_k)$, we have
\begin{align}\label{eq:27}
f(y_k) \le f(\hat y) + \l \mu_k, y_k-\hat y\r.
\end{align}
By using (\ref{eq:18}), it yields
\begin{align*}
\l \mu_k, y_k-\hat y\r
& \le \l \mu_1, y_k-\hat y\r + \sum_{i=2}^{k} E_i
+ \rho_2 \|y_{k-1}-y_k\| \|W(x_k-\hat x)\| \\
& \quad \,  + \rho_2 \|y_0-y_1\| \|W(x_k-\hat x)\|,
\end{align*}
which together with (\ref{eq:27}) and Lemma \ref{lem4} implies that $f(y_k) \le C<\infty$ for some
constant $C$ independent of $k$. By the lower semi-continuity of $f$ we have
\begin{align}\label{eq:6.25}
f(\tilde{y}) \le \liminf_{k\rightarrow \infty} f(y_k) \le C<\infty.
\end{align}
Thus $\tilde{y}\in \D(f)$. Since the above argument shows that$(\tilde x, \tilde y)$ is a feasible point
of (\ref{eq:2}), we may replace $(\hat x, \hat y)$ in (\ref{eq:18}) by $(\tilde x,\tilde{y})$ and use
$y_k \rightarrow \tilde y$ and $W x_k \rightarrow W \tilde x$ to obtain
$$
\limsup_{k\rightarrow \infty} |\l \mu_k, y_k-\tilde{y}\r| \le \sum_{i=m+1}^\infty E_i
$$
for all integers $m$. This together with Lemma \ref{lem4} implies that $\l \mu_k, y_k-\tilde{y}\r\rightarrow 0$
as $k\rightarrow \infty$. Now we can use (\ref{eq:27}) with $\hat y$ replaced by $\tilde{y}$ to obtain
$
\limsup_{k\rightarrow \infty} f(y_k) \le f(\tilde{y}).
$
This together with (\ref{eq:6.25}) gives $\lim_{k\rightarrow \infty} f(y_k) = f(\tilde{y})$.
It is now straightforward to show that
$\lim_{k\rightarrow \infty} D_{\mu_k} f(\tilde{y}, y_k) =0$.

Finally we show that $\tilde x = x^*$ and $\tilde{y} = y^*$. To see this, we first prove that
$f(\tilde{y}) \le f(\hat y)$ for any feasible point $(\hat x, \hat y)$ of (\ref{eq:2}). We will use (\ref{eq:27}).
Let $\varepsilon>0$ be any small number. By using (\ref{eq:33}) and Lemma \ref{lem4}, we can find $k_0$ such that
\begin{align}\label{eq:21}
|\l \mu_k-\mu_{k_0}, y_k-\hat y\r| \le \varepsilon \quad \mbox{and} \quad
\rho_2 |\l y_{k_0-1}-y_{k_0}, W(x_k-\hat x)\r| \le \varepsilon
\end{align}
for all $k\ge k_0$. Thus, it follows from (\ref{eq:27}) that
\begin{align}\label{eq:34}
f(y_k) \le f(\hat y) + \varepsilon + \l \mu_{k_0}, y_k-\hat y\r.
\end{align}
In view of (\ref{eq:10}) and Lemma \ref{lem1} we have
\begin{align*}
\l \mu_{k_0}, y_k-\hat y\r &= -\l \mu_{k_0}, s_k\r +\l \mu_{k_0}, W(x_k-\hat x)\r\\
& = -\l \mu_{k_0}, s_k\r + \l \mu_1, W(x_k-\hat x)\r + \rho_2 \sum_{i=2}^{k_0} \l s_i, W(x_k-\hat x)\r \\
&= -\l \mu_{k_0}, s_k\r - \l \la_1, r_k\r +\rho_2 \l y_0-y_1, W(x_k-\hat x)\r\\
& \quad \ - \sum_{i=2}^{k_0} \left( \rho_1\l r_i, r_k\r - \rho_2 \l (y_{i-1}-y_i)-(y_{i-2}-y_{i-1}),W(x_k-\hat x)\r\right) \\
& =-\l \mu_{k_0}, s_k\r  -\l \la_1, r_k\r -\rho_1 \sum_{i=2}^{k_0} \l r_i, r_k\r + \rho_2 \l y_{k_0-1}- y_{k_0}, W(x_k-\hat x)\r.
\end{align*}
Thus, by using the second equation in (\ref{eq:21}) we can derive that
\begin{align}\label{eq:26}
|\l \mu_{k_0}, y_k-\hat y\r| \le \|\mu_{k_0}\|\|s_k\| +\|\la_1\| \|r_k\|
+ \rho_1 \left(\sum_{i=2}^{k_0}\|r_i\|\right) \|r_k\| + \varepsilon.
\end{align}
Combining (\ref{eq:34}) and (\ref{eq:26}), we obtain
\begin{align*}
f(y_k) \le f(\hat y) + 2\varepsilon + \|\mu_{k_0}\|\|s_k\| + \|\la_1\| \|r_k\|
+ \rho_1 \left(\sum_{i=2}^{k_0}\|r_i\|\right) \|r_k\|.
\end{align*}
In view of Lemma \ref{lem4}, this implies that
$
\limsup_{k\rightarrow \infty} f(y_k) \le f(\hat y) +2 \varepsilon.
$
By the lower semi-continuity of $f$ and the fact $y_k\rightarrow \tilde{y}$ we can derive that
$$
f(\tilde{y}) \le \liminf_{k\rightarrow \infty} f(y_k) \le f(\hat y) + 2\varepsilon.
$$
Because $\varepsilon>0$ can be arbitrarily small, we must have $f(\tilde{y}) \le f(\hat y)$ for any feasible point
$(\hat x, \hat y)$ of (\ref{eq:2}). Since $(\tilde x, \tilde y)$ is a feasible point of (\ref{eq:2}), it follows that
$$
f(W \tilde x) = f(\tilde{y}) = f(y^*) = f(W x^*)= \min\left\{f(W x): x\in \D(W) \mbox{ and } A x =b\right\}.
$$
From the uniqueness of $x^*$, see Theorem \ref{thm:exist},
we can conclude that $\tilde{x} = x^*$ and hence $\tilde{y} = y^*$.  The proof is therefore complete.
\end{proof}

\subsection{Noisy data case}

In practical applications, the data are usually obtained by measurement and unavoidably
contain error. Thus, instead of $b$, usually we only have noisy data $b^\d$ satisfying
$$
\|b^\d-b\|\le \d
$$
for a small noise level $\d>0$. In this situation, we need to replace $b$ in our ADMM algorithm
(\ref{eq:3})--(\ref{eq:6}) by the noisy data $b^\d$ for numerical computation. For inverse problems,
an iterative method using noisy data usually exhibits semi-convergence property, i.e.
the iterate converges toward the sought solution at the beginning, and, after a critical number of iterations,
the iterate eventually diverges from the sought solution due to the amplification of noise.
The iteration should be terminated properly in order to produce a reasonable approximate solution for (\ref{eq:1}).
Incorporating a stopping criterion into the iteration leads us to propose Algorithm \ref{alg:noise}
for solving inverse problems with noisy data.

\begin{algorithm}[H]
   \caption{(ADMM with noisy data)}\label{alg:noise}
   \begin{algorithmic}[1]
     \STATE Input: initial guess $y_0\in \Y$, $\la_0\in \H$ and $\mu_0\in \Y$, constants $\rho_1>0$, $\rho_2>0$
     and $\tau>1$, noise level $\delta>0$.
     \STATE Let $y_0^\d = y_0$, $\la_0^\d =\la_0$ and $\mu_0^\d = \mu_0$.
     \FOR {$k=0,1, \cdots$}
     \STATE update $x$, $y$ and the Lagrange multipliers $\la$, $\mu$ as follows:
      \begin{align*}
x_{k+1}^\d & = \arg\min_{x\in \D(W)} \left\{\l \la_k^\d, A x\r + \l \mu_k^\d, Wx\r
   + \frac{\rho_1}{2} \|A x-b^\d\|^2 +\frac{\rho_2}{2} \|Wx-y_k^\d\|^2\right\} ,\\
y_{k+1}^\d  &= \arg\min_{y\in \Y} \left\{ f(y) - \l \mu_k^\d, y\r + \frac{\rho_2}{2} \|Wx_{k+1}^\d -y\|^2 \right\},\\
\la_{k+1}^\d &= \la_k^\d + \rho_1 (A x_{k+1}^\d -b^\d),\\
\mu_{k+1}^\d &= \mu_k^\d + \rho_2 (Wx_{k+1}^\d -y_{k+1}^\d).
\end{align*}
     \STATE check the stopping criterion:
     \begin{equation}\label{eq:stop}
     \rho_1^2\|A x_k^\d -b^\d\|^2 +\rho_2^2 \|Wx_k^\d -y_k^\d\|^2 \le \max(\rho_1^2,\rho_2^2)\tau^2 \d^2.
     \end{equation}
     \ENDFOR
   \end{algorithmic}
\end{algorithm}

Under (\textbf{A1})--(\textbf{A4}), we may use Lemma \ref{lem:w2} to conclude that $x_k^\d$, $y_k^\d$,
$\la_k^\d$ and $\mu_k^\d$ in Algorithm \ref{alg:noise} are well-defined for $k\ge 1$. Furthermore, we
have the following stability result in which we take $x_0^\d =x_0$ to be any element in $\D(W)$.

\begin{lemma}\label{lem:stability}
Consider Algorithm \ref{alg:noise} without \eqref{eq:stop}. Then for each fixed $k\ge 0$ there hold
$$
x_k^\d \rightarrow x_k, \quad y_k^\d\rightarrow y_k, \quad W x_k^\d \rightarrow W x_k, \quad
\la_k^\d \rightarrow \la_k, \quad \mu_k^\d\rightarrow \mu_k,\quad
f(y_k^\d) \rightarrow f(y_k)
$$
as $\d\rightarrow 0$, where $(x_k,y_k,\lambda_k,\mu_k)$ are defined by the ADMM algorithm (\ref{eq:3})--(\ref{eq:6})
with exact data.
\end{lemma}

\begin{proof}
We use an induction argument. The result is trivial when $k=0$. Assuming that the result is true for some $k=n$,
we show that it is also true for $k=n+1$. From Lemma \ref{lem:w2} and the induction hypothesis we can obtain that
$$
x_{n+1}^\d \rightarrow x_{n+1}, \quad y_{n+1}^\d \rightarrow y_{n+1}, \quad W x_{n+1}^\d \rightarrow W x_{n+1}
\quad \mbox{and} \quad  f(y_{n+1}^\d)\rightarrow f(y_{n+1})
$$
as $\d\rightarrow 0$. Now we can obtain $\la_{n+1}^\d\rightarrow \la_{n+1}$ and $\mu_{n+1}^\d\rightarrow \mu_{n+1}$
as $\d\rightarrow 0$ from their definition.
\end{proof}

In the following we will show that Algorithm \ref{alg:noise} terminates after a finite number of iterations
and defines a regularization method. For simplicity of exposition, we use the notation
$$
r_k^\d = A x_k^\d - b^\d \quad \mbox{and} \quad s_k^\d = Wx_k^\d - y_k^\d.
$$
From the description of Algorithm \ref{alg:noise}, one can easily see that
\begin{align*}
\la_{k+1}^\d -\la_k^\d  = \rho_1 r_{k+1}^\d, & \qquad \mu_{k+1}^\d  \in \p f(y_{k+1}^\d),\\
\mu_{k+1}^\d -\mu_k^\d  = \rho_2 s_{k+1}^\d, & \qquad
A^* \la_k^\d + \rho_1A^* r_{k+1}^\d = - W^*[\mu_k^\d+ \rho_2 (Wx_{k+1}^\d-y_k^\d)]
\end{align*}
for $k\ge 0$. By the same argument in the proof of Lemma \ref{lem1} one can derive that
\begin{align}\label{eq:616}
\rho_1\l r_{k+1}^\d, A x\r = \rho_2\left\l (y_k^\d-y_{k+1}^\d) -(y_{k-1}^\d-y_k^\d)-s_{k+1}^\d, W x\right\r, \quad x\in \D(W).
\end{align}
for $k\ge 1$. Furthermore, by using the same argument in the proof of Lemma \ref{lem3}
we can show the following result.

\begin{lemma}\label{lem:7}
Let $E_k^\d= \rho_1\|A x_k^\d-b^\d\|^2 +\rho_2 \|Wx_k^\d-y_k^\d\|^2 +\rho_2 \|y_k^\d-y_{k-1}^\d\|^2$ for $k\ge 1$.
Then
$$
E_{k+1}^\d -E_k^\d \le -\rho_1\|A (x_{k+1}^\d-x_k^\d)\|^2 -4c_0 \|y_{k+1}^\d-y_k^\d\|^2.
$$
Consequently $\{E_k^\d\}$ is monotonically decreasing along the iteration and there hold
$$
\sum_{k=m}^{n-1} \|y_{k+1}^\d-y_k^\d\|^2 \le \frac{1}{4 c_0} E_m^\d \quad \mbox{and}\quad
(n-m) \rho_2\|y_n^\d -y_{n-1}^\d\|^2 \le \sum_{k=m+1}^n E_k^\d
$$
for any integers $1\le m<n$.
\end{lemma}

The following result shows that the stoping criterion (\ref{eq:stop}) is satisfied for some finite integer,
hence  Algorithm \ref{alg:noise}  terminates  after a finite number of iterations.

\begin{lemma}\label{lem:10}
There exists a finite integer $k_\d$ such that the stop condition \eqref{eq:stop} is satisfied for the first time.
Moreover, there exist positive constants $c$ and $C$ depending only on $\rho_2$, $\tau$ and $c_0$ such that
\begin{align}\label{eq:59}
D_{\mu_n^\d} f(\hat{y}, y_n^\d) + c \sum_{k=m}^n E_k^\d
& \le D_{\mu_m^\d} f(\hat{y}, y_m^\d) + \rho_2 \left\l y_{m-1}^\d-y_m^\d,W( \hat{x}-x_{m+1}^\d)\right\r \nonumber\\
& \quad \, + C \left(\|W(\hat{x} -x_m^\d)\|^2 + \|s_m^\d\|^2 + E_m^\d\right)
\end{align}
for any integers $1\le m<n<k_\d$ and any feasible point $(\hat{x}, \hat{y})$ of (\ref{eq:2}).
\end{lemma}

\begin{proof}
Similar to the derivation of (\ref{eq:28}) and using (\ref{eq:616}) we can obtain for $k\ge 1$ that
\begin{align}\label{eq:61}
& D_{\mu_{k+1}^\d} f(\hat{y}, y_{k+1}^\d) - D_{\mu_k^\d} f(\hat{y}, y_k^\d)+ D_{\mu_k^\d} f(y_{k+1}^\d, y_k^\d) \nonumber\\
& = -\rho_2 \|s_{k+1}^\d\|^2 -\rho_1\|r_{k+1}^\d\|^2 + \rho_1 \left\l r_{k+1}^\d, b-b^\d \right\r
+ \rho_2 \left\l y_{k-1}^\d-y_k^\d, W(\hat{x}-x_{k+1}^\d)\right \r \nonumber\\
& \quad \, -\rho_2\left\l y_k^\d-y_{k+1}^\d, W(\hat{x} - x_{k+1}^\d) \right\r \displaybreak[0] \nonumber\\
& \le -\rho_2 \|s_{k+1}^\d\|^2 -\rho_1\|r_{k+1}^\d\|^2 + \rho_1\d \| r_{k+1}^\d\|
+ \rho_2 \left\l y_{k-1}^\d-y_k^\d, W(\hat{x}-x_{k+1}^\d)\right\r \nonumber\\
& \quad \, -\rho_2 \left\l y_k^\d-y_{k+1}^\d, W(\hat{x} - x_{k+1}^\d) \right\r.
\end{align}
By using the formulation of the stop criterion \eqref{eq:stop}, we can see that for any $k$
satisfying $1\le k <k_\d-1$ there holds
\begin{align*}
\rho _1\d \| r_{k+1}^\d\| &\le \frac{\rho_1^2 \|r_{k+1}^\d\|^2 + \rho_2^2\|s_{k+1}^\d\|^2}{\tau \max(\rho_1,\rho_2)}
\le \frac{1}{\tau} \left(\rho_1 \|r_{k+1}^\d\|^2 + \rho_2 \|s_{k+1}^\d\|^2 \right).
\end{align*}
By the strong convexity of $f$ we have $D_{\mu_k^\d} f(y_{k+1}^\d, y_k^\d) \geq c_0 \|y_k^\d - y_{k+1}^\d\|^2$.
Therefore, by setting
$
c_2 := \min \left\{1 - 1/\tau, c_0/\rho_2\right\} >0
$
and using the definition of $E_k^\d$ we have
\begin{align}\label{eq:48}
& D_{\mu_{k+1}^\d} f(\hat{y}, y_{k+1}^\d) - D_{\mu_k^\d} f(\hat{y}, y_k^\d)+ c_2 E_{k+1}^\d \nonumber\\
& \le  \rho_2 \left\l y_{k-1}^\d-y_k^\d, W(\hat{x}-x_{k+1}^\d)\right\r
-\rho_2 \left\l y_k^\d-y_{k+1}^\d,W( \hat{x} - x_{k+1}^\d)\right\r.
\end{align}
For any two integers $1\le m<n<k_\d$, we sum (\ref{eq:48}) over $k$ from $k=m$ to $k=n-1$ to derive that
\begin{align}\label{eq:62}
& D_{\mu_n^\d} f(\hat{y}, y_n^\d)  + c_2 \sum_{k=m+1}^{n} E_k^\d \nonumber\\
& \le D_{\mu_m^\d} f(\hat{y}, y_m^\d)  + \rho_2 \left\l y_{m-1}^\d-y_m^\d, W(\hat{x}-x_{m+1}^\d) \right\r
 - \rho_2 \left\l y_{n-1}^\d-y_n^\d, W(\hat{x} - x_n^\d) \right\r \nonumber \\
& \quad \, + \rho_2\sum_{k=m}^{n-2} \left\l y_k^\d-y_{k+1}^\d, W(x_{k+1}^\d-x_{k+2}^\d) \right\r.
\end{align}
Let $\varepsilon>0$ be a small number which will be specified later. By using the Cauchy-Schwarz inequality
and the similar arguments for deriving (\ref{eq:29}) and (\ref{eq:299}), we can show that there is
constant $C_\varepsilon>0$  depending only on $\varepsilon$ such that
\begin{align*}
& \sum_{k=m}^{n-2} \left\l y_k^\d-y_{k+1}^\d, W(x_{k+1}^\d-x_{k+2}^\d)\right\r
\le \varepsilon \sum_{k=m+1}^n \|s_k^\d\|^2 + C_\varepsilon \sum_{k=m}^{n-1} \|y_k^\d-y_{k+1}^\d\|^2
\end{align*}
and
\begin{align}\label{eq:69}
- \left\l y_{n-1}^\d-y_n^\d, W(\hat{x} - x_n^\d) \right\r
& \le  \varepsilon \left(\|s_m^\d\|^2 + \|s_n^\d\|^2\right) + \varepsilon (n-m) \|y_{n-1}^\d-y_n^\d\|^2 \nonumber\\
& \quad \, + \frac{1}{4} \|W(\hat{x}-x_m^\d)\|^2 + C_\varepsilon \sum_{k=m}^{n-1} \|y_k^\d-y_{k+1}^\d\|^2.
\end{align}
Combining the above two equations with (\ref{eq:62}), we obtain
\begin{align*}
& D_{\mu_n^\d} f(\hat{y}, y_n^\d) + c_2 \sum_{k=m+1}^{n} E_k^\d   \nonumber\\
& \le D_{\mu_m^\d} f(\hat{y}, y_m^\d)  + \rho_2 \left\l y_{m-1}^\d-y_m^\d,W( \hat{x}-x_{m+1}^\d)\right\r
+ \frac{\rho_2}{4} \|W(\hat{x} - x_m^\d)\|^2 + \varepsilon \rho_2\|s_m^\d\|^2 \nonumber\\
& \quad \, + 2 \varepsilon \rho_2\sum_{k=m+1}^n \|s_k^\d\|^2 + \varepsilon (n-m) \rho_2 \|y_{n-1}^\d -y_n^\d\|^2
+ 2 \rho_2 C_\varepsilon \sum_{k=m}^{n-1} \|y_k^\d-y_{k+1}^\d\|^2.
\end{align*}
By using Lemma \ref{lem:7} we further obtain
\begin{align*}
D_{\mu_n^\d} f(\hat{y}, y_n^\d) & + c_2 \sum_{k=m+1}^{n} E_k^\d
\le D_{\mu_m^\d} f(\hat y, y_m^\d)  + \rho_2 \left\l y_{m-1}^\d-y_m^\d, W(\hat{x}-x_{m+1}^\d)\right\r\\
& + \frac{\rho_2}{4} \|W(\hat{x} - x_m^\d)\|^2 + \varepsilon \rho_2\|s_m^\d\|^2 + 3 \varepsilon \sum_{k=m+1}^n E_k^\d
+ \frac{\rho_2 C_\varepsilon}{2 c_0} E_m^\d.
\end{align*}
Now we take $\varepsilon = c_2/6 $. Then
\begin{align}\label{eq:63}
D_{\mu_n^\d} f(\hat{y}, y_n^\d) + \frac{c_2}{2}  \sum_{k=m+1}^{n} E_k^\d
& \le D_{\mu_m^\d} f(\hat{y}, y_m^\d)  + \rho_2 \left\l y_{m-1}^\d-y_m^\d, W(\hat{x}-x_{m+1}^\d) \right\r \nonumber\\
& + \frac{\rho_2}{4} \|W(\hat{x} - x_m^\d)\|^2 + \varepsilon \rho_2\|s_m^\d\|^2 + \frac{\rho_2 C_\varepsilon}{2 c_0} E_m^\d.
\end{align}
This shows (\ref{eq:59}) immediately.

Finally we show that there is a finite integer $k_\d$ such that (\ref{eq:stop}) is
satisfied. If not, then for any $k\ge 1$ there holds
$$
\rho_1^2\|r_k^\d\|^2+\rho_2^2\|s_k^\d\|^2 >\max(\rho_1^2,\rho_2^2)\tau^2 \d^2
$$
It then follows from (\ref{eq:59}) with $m=1$ that
\begin{align} \label{eq:tmp1}
c (n-1)\max(\rho_1,\rho_2) \tau^2 \d^2 \le  c \sum_{k=2}^n E_k^\d
& \le D_{\mu_1^\d} f(\hat{y}, y_1^\d)  + \rho_2 \left\l y_0^\d-y_1^\d, W(\hat{x}-x_2^\d)\right\r \nonumber\\
& + C\left(\|W(\hat{x} - x_1^\d)\|^2 + \|s_1^\d\|^2 +  E_1^\d\right).
\end{align}
for any integer $n\ge 1$. Letting $n\rightarrow \infty$ yields a contradiction.
\end{proof}

\begin{remark}
{\rm Let $k_\d$ denote the first integer such that (\ref{eq:stop}) is satisfied. Then \eqref{eq:tmp1} holds
for all $n<k_\d$. According to Lemma \ref{lem:stability}, the right hand side of (\ref{eq:tmp1}) can be
bounded by a constant independent of $\d$. Thus, we may use it to conclude that $k_\d = O(\d^{-2})$.
}
\end{remark}

We next derive
some estimates which will be crucially used in the proof of regularization property of Algorithm \ref{alg:noise}.

\begin{lemma}\label{lem:11}
There exist positive constants $c$ and $C$ depending only on $\rho_2$, $\tau$ and $c_0$ such that
for any integer $m<k_\d-1$ there hold
\begin{align*}
D_{\mu_{k_\d}^\d} f(\hat{y}, y_{k_\d}^\d) + c E_{k_\d}^\d
& \le D_{\mu_m^\d} f(\hat{y}, y_m^\d) + \max\{\rho_1, \rho_2\} \tau \d^2 + C\|W(\hat{x} - x_m^\d)\|^2+ CE_m^\d  \nonumber\\
& \quad \, + C\|s_m^\d\|^2  + C\left| \left\l y_{m-1}^\d-y_m^\d, W(\hat{x}-x_{m+1}^\d)\right\r \right|
\end{align*}
and
\begin{align*}
\left|\left\l \mu_{k_\d}^\d, y_{k_\d}^\d- \hat{y}\right\r \right|
\le \left|\left\l \mu_m^\d, y_{k_\d}^\d -\hat{y}\right\r\right|
+ \max\{\rho_1, \rho_2\} \tau \d^2 + C \sum_{k=m}^{k_\d} E_k^\d + C \|W(x_{k_\d}^\d -\hat{x})\|^2,
\end{align*}
where $(\hat x, \hat y)$ denotes any feasible point of (\ref{eq:2}).
\end{lemma}

\begin{proof}
By using (\ref{eq:61}) with $k = k_\d-1$, the strong convexity of $f$, and the fact $\rho_1\|r_{k_\d}^\d\| \le \max(\rho_1, \rho_2)\tau \d$, we have
\begin{align*}
D_{\mu_{k_\d}^\d} f(\hat{y}, y_{k_\d}^\d) + c_0 \|y_{k_\d}^\d -y_{k_\d-1}^\d\|^2
& \le D_{\mu_{k_\d-1}^\d} f(\hat{y}, y_{k_\d-1}^\d) -\rho_2 \|s_{k_\d}^\d\|^2 \\
& -\rho_1\|r_{k_\d}^\d\|^2 + \max\{\rho_1, \rho_2\} \tau \d^2 \\
& + \rho_2 \left\l y_{k_\d-2}^\d-y_{k_\d-1}^\d, W(\hat{x}-x_{k_\d}^\d)\right\r  \\
& -\rho_2 \left\l y_{k_\d-1}^\d-y_{k_\d}^\d, W(\hat{x} - x_{k_\d}^\d) \right\r.
\end{align*}
Let $\varepsilon>0$ be a  small number specified later. Similar to (\ref{eq:299}) and (\ref{eq:69}) we can derive for $m<k_\d-1$ that
\begin{align*}
&\left\l y_{k_\d-2}^\d-y_{k_\d-1}^\d, W(\hat{x}-x_{k_\d}^\d)\right\r
-\left\l y_{k_\d-1}^\d-y_{k_\d}^\d, W(\hat{x} - x_{k_\d}^\d)\right\r \\
& \le \frac{1}{2} \|W(\hat{x}-x_m^\d)\|^2 +\varepsilon \|s_m^\d\|^2 +\frac{\varepsilon}{\rho_2}\|E_{k_\d}^\d\|^2
   + \frac{2 \varepsilon}{\rho_2} \sum_{k=m+1}^{k_\d-1} E_k^\d + \frac{C_\varepsilon}{4 c_0} E_m^\d.
\end{align*}
Therefore
\begin{align*}
& D_{\mu_{k_\d}^\d} f(\hat{y}, y_{k_\d}^\d) + c_0 \|y_{k_\d}^\d -y_{k_\d-1}^\d\|^2+\rho_1 \|r_{k_\d}^\d\|^2
+\rho_2\|s_{k_\d}^\d\|^2 \displaybreak[0]\\
& \le D_{\mu_{k_\d-1}^\d} f(\hat{y}, y_{k_\d-1}^\d)  + \max\{\rho_1, \rho_2\} \tau \d^2
+ \frac{\rho_2}{2} \|W(\hat x-x_m^\d)\|^2 +\varepsilon \rho_2\|s_m^\d\|^2 \displaybreak[0]\\
& \quad \, +\varepsilon\|E_{k_\d}^\d\|^2 + 2 \varepsilon \sum_{k=m+1}^{k_\d-1} E_k^\d + \frac{\rho_2C_\varepsilon}{4c_0} E_m^\d.
\end{align*}
By taking $\varepsilon = \min\{1,c_0/\rho_2\}/2$, we obtain with a constant $c_3 = \varepsilon$ that
\begin{align}\label{eq:66}
D_{\mu_{k_\d}^\d} f(\hat{y}, y_{k_\d}^\d) + c_3 E_{k_\d}^\d
& \le  D_{\mu_{k_\d-1}^\d} f(\hat{y}, y_{k_\d-1}^\d)  + \max\{\rho_1, \rho_2\} \tau \d^2
+ \frac{\rho_2}{2}\|W(\hat x-x_m^\d)\|^2 \nonumber \\
& \quad \, +\varepsilon \rho_2\|s_m^\d\|^2 + 2 \varepsilon \sum_{k=m+1}^{k_\d-1} E_k^\d +\frac{\rho_2 C_\varepsilon}{4c_0} E_m^\d.
\end{align}
An application of (\ref{eq:59}) with $n= k_\d-1$ then gives the first estimate.

To see the second one, we apply similar argument for deriving (\ref{eq:18}) and the Cauchy-Schwarz inequality to obtain
\begin{align*}
&\left| \left\l \mu_{k_\d}^\d - \mu_m^\d, y_{k_\d}^\d - \hat{y}\right\r \right|\\
& \le  \sum_{k=m+1}^{k_\d} E_k^\d  + \rho_1 \left|\sum_{k=m+1}^{k_\d} \left\l r_{k}^\d, b^\d-b \right\r\right|
+\rho_2\left|\left\l y_{k_\d-1}^\d -y_{k_\d}^\d, W(x_{k_\d}^\d -\hat{x})\right\r \right| \displaybreak[0]\\
& \quad \, +\rho_2 \left| \left\l y_{m-1}^\d -y_m^\d, W(x_{k_\d}^\d -\hat{x}) \right\r\right| \displaybreak[0]\\
&  \le  \sum_{k=m+1}^{k_\d} E_k^\d  + \rho_1\d \sum_{k=m+1}^{k_\d} \| r_k^\d\|
+E_{k_\d}^\d +  E_m^\d + \frac{\rho_2}{2} \|W(x_{k_\d}^\d -\hat x)\|^2.
\end{align*}
Note that $\rho_1\|r_{k_\d}^\d\| \le \max(\rho_1, \rho_2)\tau \d$ and $\max(\rho_1^2, \rho_2^2)\tau^2 \d^2
\le \rho_1^2\|r_k^\d\|^2 +\rho_2^2 \|s_k^\d\|^2$ for $k<k_\d$. We thus obtain
\begin{align*}
\rho_1\d \sum_{k=m+1}^{k_\d} \|r_k^\d\|
 = \rho_1 \d\|r_{k_\d}^\d\| +  \sum_{k=m+1}^{k_\d-1} \rho_1 \d \|r_k^\d\|
\le \max\{\rho_1, \rho_2\} \tau \d^2 + \frac{1}{\tau} \sum_{k=m+1}^{k_\d-1}  E_k^\d.
\end{align*}
Therefore
\begin{align*}
\left| \left\l \mu_{k_\d}^\d - \mu_m^\d, y_{k_\d}^\d - \hat{y} \right\r \right|
\le \left(2+ \frac{1}{\tau}\right) \sum_{k=m}^{k_\d} E_k^\d  + \max\{\rho_1, \rho_2\} \tau \d^2
+\frac{\rho_2}{2} \|W(x_{k_\d}^\d - \hat{x})\|^2
\end{align*}
which gives the desired estimate.
\end{proof}

\begin{theorem}
Let \emph{(\textbf{A1})--(\textbf{A4})} hold and let $b$ be consistent. Let $x^*$ be the unique solution of (\ref{eq:1})
and let $y^* = W x^*$. Let $k_\d$ denote the first integer such that (\ref{eq:stop}) is satisfied.
Then for Algorithm \ref{alg:noise} there hold
$$
x_{k_\d}^\d \rightarrow x^*, \quad y_{k_\d}^\d\rightarrow y^*,\quad W x_{k_\d}^\d \rightarrow y^*,
\quad f(y_{k_\d}^\d)\rightarrow f(y^*), \quad D_{\mu_{k_\d}^\d} f(y^*, y_{k_\d}^\d) \rightarrow 0
$$
as $\d \rightarrow 0$.
\end{theorem}

\begin{proof}
We show the convergence result by considering two cases via a subsequence-subsequence argument.

Assume first that $\{b^{\d_i}\}$ is a sequence satisfying $\|b^{\d_i}-b\|\le \d_i$ with $\d_i\rightarrow 0$
such that $k_{\d_i} = k_0$ for all $i$, where $k_0$ is a finite integer. By the definition of $k_{\d_i}$ we have
$$
\rho_1^2\|A x_{k_0}^{\d_i} - b^{\d_i}\|^2 +\rho_2^2\|Wx_{k_0}^{\d_i} - y_{k_0}^{\d_i}\|^2 \le \max(\rho_1^2, \rho_2^2)\tau^2 \d_i^2.
$$
Letting $i\rightarrow \infty$ and using Lemma \ref{lem:stability}, we can obtain
$A x_{k_0} = b$ and $Wx_{k_0} = y_{k_0}$. This together with the definition of $\la_k$ and $\mu_k$ implies
that $\la_{k_0}=\la_{k_0-1}$ and $\mu_{k_0}=\mu_{k_0-1}$.
Recall that $\mu_k\in \p f(y_k)$, we may use (\ref{eq:77}) to obtain
$$
0 = \l \mu_{k_0}-\mu_{k_0-1}, y_{k_0}-y_{k_0-1}\r \ge 2 c_0 \|y_{k_0}-y_{k_0-1}\|^2
$$
which implies $y_{k_0}=y_{k_0-1}$. Now we can use (\ref{eq:3}) and (\ref{eq:4}) to conclude that $x_{k_0+1} = x_{k_0}$
and $y_{k_0+1}=y_{k_0}$. Repeating this argument we can derive that $x_k= x_{k_0}$, $y_k=y_{k_0}$, $\la_k=\la_{k_0}$
and $\mu_k=\mu_{k_0}$ for all $k\ge k_0$. In view of Theorem \ref{thm:exact}, we must have $x_{k_0}= x^*$ and $y_{k_0}=y^*$.
With the help of Lemma \ref{lem:stability}, the desired conclusion then follows.

Assume next that $\{b^{\d_i}\}$ is a sequence satisfying $\|b^{\d_i}-b\|\le \d_i$ with $\d_i\rightarrow 0$
such that $k_i:=k_{\d_i} \rightarrow \infty$ as $i\rightarrow \infty$. We first show that
\begin{align}\label{eq:64}
D_{\mu_{k_i}^{\d_i}} f(y^*, y_{k_i}^{\d_i}) \rightarrow 0 \quad \mbox{ as } i\rightarrow \infty.
\end{align}
Let $m\ge 1$ be any integer. Then $k_i> m+1$ for large $i$. Thus we may use Lemma \ref{lem:11} to conclude that
\begin{align*}
D_{\mu_{k_i}^{\d_i}} f(y^*, y_{k_i}^{\d_i})
& \le D_{\mu_m^{\d_i}} f(y^*, y_m^{\d_i}) + \max\{\rho_1, \rho_2\} \tau \d_i^2 + C\|W(x^* - x_m^{\d_i})\|^2 + C\|s_m^{\d_i}\|^2 \\
& \quad \, + C E_m^{\d_i} + C\left| \l y_{m-1}^{\d_i}-y_m^{\d_i}, W(x^*-x_{m+1}^{\d_i}) \r \right|.
\end{align*}
By virtue of Lemma \ref{lem:stability}, we have
\begin{align*}
\limsup_{i\rightarrow \infty} D_{\mu_{k_i}^{\d_i}} f(y^*, y_{k_i}^{\d_i})
& \le D_{\mu_m} f(y^*, y_m) + C\|W(x^* -x_m)\|^2 + C\|s_m\|^2 + C E_m\\
& \quad \, + C \left|\l y_{m-1}-y_m, W(x^*-x_{m+1})\r\right|.
\end{align*}
Letting $m\rightarrow \infty$ and using Theorem \ref{thm:exact}, we obtain
$$
\limsup_{i\rightarrow \infty} D_{\mu_{k_i}^{\d_i}} f(y^*, y_{k_i}^{\d_i}) \le 0
$$
which shows (\ref{eq:64}). Now by using the strong convexity of $f$ we can conclude that
$y_{k_i}^{\d_i} \rightarrow y^*$ as $i\rightarrow \infty$. Since $\rho_1^2\|Ax_{k_i}^{\d_i} - b^{\d_i}\|
+ \rho_2^2\|Wx_{k_i}^{\d_i}-y_{k_i}^{\d_i}\|^2 \le \max(\rho_1^2, \rho_2^2)\tau^2 \d_i^2$,
we also have $A x_{k_i}^{\d_i} \rightarrow b$ and $W x_{k_i}^{\d_i} \rightarrow y^*$ as $i \rightarrow \infty$.
In view of (\textbf{A4}), we have
$$
c_1 \|x_{k_i}^{\d_i} - x^*\| \le \|A x_{k_i}^{\d_i} - b\|^2 + \|W x_{k_i}^{\d_i} - y^*\|^2
$$
which implies that $x_{k_i}^{\d_i} \rightarrow x^*$ as $i\rightarrow \infty$.

Finally, we show that $f(y_{k_i}^{\d_i}) \rightarrow f(y^*)$ as $i\rightarrow\infty$.
According to (\ref{eq:64}), it suffices to show that
\begin{align}\label{eq:81}
\l u_{k_i}^{\d_i}, y^* - y_{k_i}^{\d_i}\r \rightarrow 0 \quad \mbox{ as } i\rightarrow \infty.
\end{align}
By virtue of Lemma \ref{lem:11}, $y_{k_i}^{\d_i} \rightarrow y^*$
and $Wx_{k_i}^{\d_i} \rightarrow Wx^*$, we have
$$
\limsup_{i\rightarrow \infty} \left|\l \mu_{k_i}^{\d_i}, y^* - y_{k_i}^{\d_i}\r\right|
\le C \limsup_{i\rightarrow \infty} \sum_{k=m}^{k_i} E_k^{\d_i}.
$$
In view of Lemma \ref{lem:10}, Lemma \ref{lem:11} and Lemma \ref{lem:stability}, we can obtain
\begin{align*}
\limsup_{i\rightarrow \infty} \left| \l  \mu_{k_i}^{\d_i}, y^* - y_{k_i}^{\d_i}\r\right|
& \le C \Big( D_{\mu_m} f(y^*, y_m) + \left|\l y_{m-1}-y_m, W(x^*-x_{m+1})\r\right| \\
& \quad \, + \|W(x^* - x_m)\|^2 + \|s_m\|^2 + E_m\Big)
\end{align*}
for any integer $m$. Letting $m\rightarrow \infty$ and using Theorem \ref{thm:exact}, it follows
\begin{align*}
\limsup_{i\rightarrow \infty} \left| \l  \mu_{k_i}^{\d_i}, y^* - y_{k_i}^{\d_i}\r\right| \le 0
\end{align*}
which shows (\ref{eq:81}). The proof is therefore complete.
\end{proof}

\section{\bf Numerical experiments}

In this section we will present various numerical results for 1-dimensional as well as 2-dimensional
problems to show the efficiency of Algorithm \ref{alg:noise}. All the experiments are done on a four-core
laptop with 1.90 GHz and 8 GB RAM. First we give the setup for the data generation, the choice of parameters
and the stopping rule. In all numerical examples the sought solutions $x^*$ are assumed to be known, and the observational
data $b^{\delta}$ is generated by $b^{\delta} = A x^* + \eta$, where $\eta$ denotes the additive measurement
noise with $ \delta = \|\eta\|$. The function $f$ is chosen as $f(\cdot) = \|\cdot\|_{*}
+ \frac{\nu}{2}\|\cdot\|^2$ for a fixed $\nu = 0.001$ (the results are not sensitive on $\nu$)
with possible different norms $\|\cdot\|_{*}$. More precisely, $\|\cdot\|_{*}$ is a weighted $\ell^1$
norm in case wavelet frame is used \cite{CaiDongOsherShen:2012}, and is the $\ell^1$ norm
for other cases. We also take the initial guess $y_0$, $\la_0$, $\mu_0$ to be the zero elements and fix $\tau = 1.0001$
in all experiments. The numerical results are not sensitive to the parameters $\rho_1$ and $\rho_2$; so we
fix them as $(\rho_1,\rho_2) = (1000,10)$ in subsections
\ref{sec:3.1}--\ref{sec:3.4}. The operators $A, W$ and the noise level $\delta$ will be
specified in each example. The ADMM codes can be found in {\tt http://xllv.whu.edu.cn/}.

\subsection{ One-dimensional deconvolution}\label{sec:3.1}

In this subsection we consider the one dimensional deconvolution problem of the form
\begin{equation}\label{deconv1d}
b^{\delta}(s) = \int_0^1 k(s,t) x(t) dt + \eta(s):=(Ax)(s) +\eta(s) \quad \mbox{on } [0,1],
\end{equation}
where
$
k(s,t) = \frac{\gamma}{\sqrt{\pi}} \exp(-\frac{(s-t)^2}{2\gamma^2})
$
with $\gamma= 0.01$. This problem arises from an inverse heat conduction.  To find the
sought solution $x^*$ from $b^\d$ numerically, we divide
$[0,1]$ into $N=400$ subintervals of equal length and approximate integrals by the midpoint
rule. Let $x^*$ be sparse, we take $W= I$ the identity. Numerical results are reported in
Figure \ref{fig:deconvld} which shows that Algorithm \ref{alg:noise} can capture the features
of solutions as the function $f$ is properly chosen. Moreover, when the noise level decreases, more
iterations are needed and more accurate approximate solutions can be obtained.
\begin{figure}[ht!]
  \centering
  \begin{tabular}{ccc}
    \includegraphics[trim = 0cm 0cm 0cm 0cm, clip=true,width=0.3\textwidth]{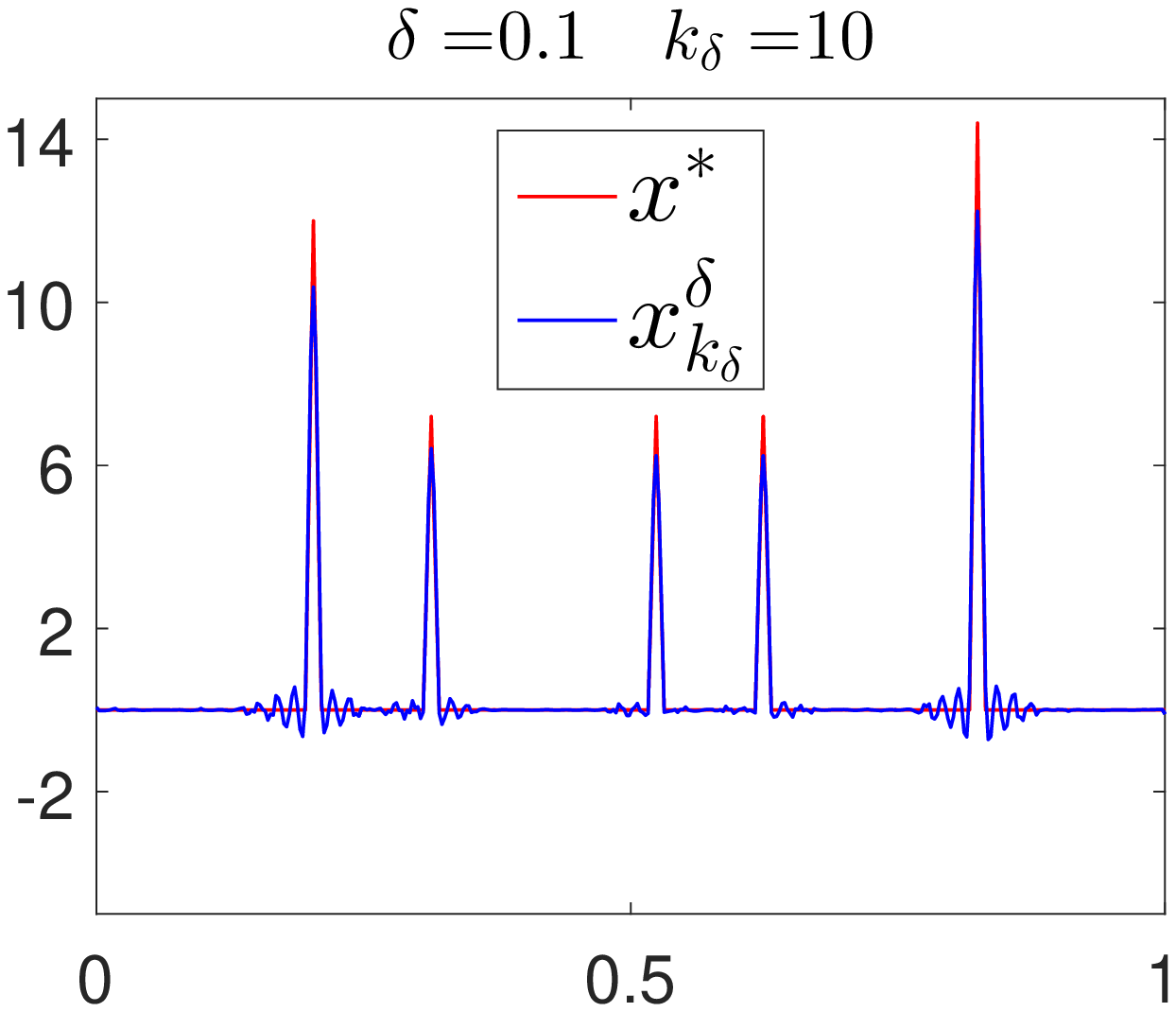}
    &\includegraphics[trim = 0cm 0cm 0cm 0cm, clip=true,width=0.3\textwidth]{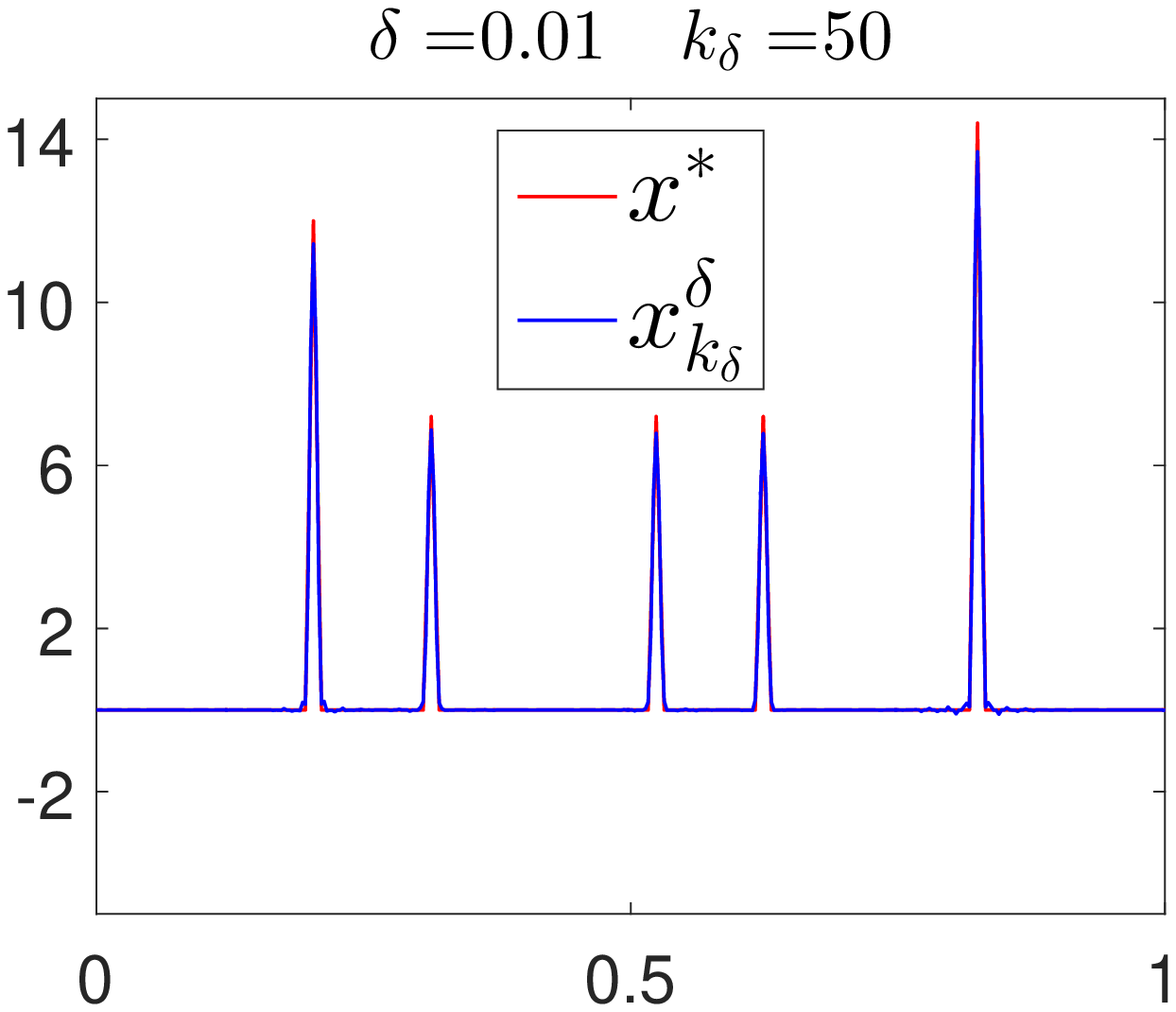}
    &\includegraphics[trim = 0cm 0cm 0cm 0cm, clip=true,width=0.3\textwidth]{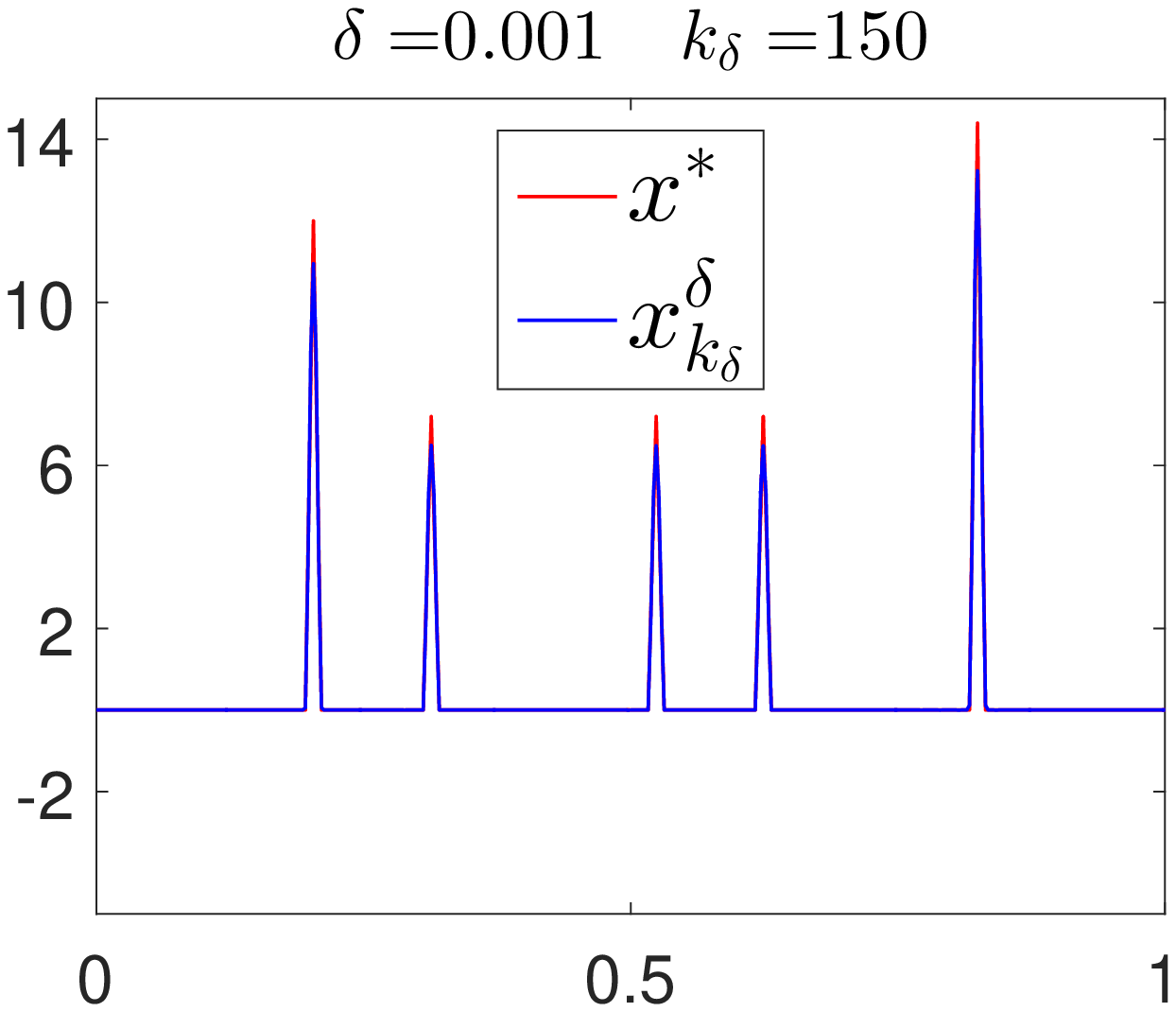} \\
  \end{tabular}
\caption{Reconstruction results for Section \ref{sec:3.1}.
}\label{fig:deconvld}
\end{figure}

\subsection{ Two-dimensional TV deblurring}

In this and next subsections we test the performance of Algorithm \ref{alg:noise} on image deblurring
problems whose objective is to reconstruct the unknown true image $x^*\in {\mathbb R}^{M\times N}$
from an observed image $b^\d = A x^* +\eta$ degraded by a linear blurring operator $A$ and a Gaussian
noise $\eta$. We consider the case that the blurring operator is shift invariant so that $A$ is a
convolution operator whose kernel is a point spread function.

This subsection concerns the total variation deblurring \cite{RudinOsherFatemi:1992}, assuming the periodic
boundary conditions on images. To apply Algorithm \ref{alg:noise}, we take $W = \nabla$ to be the
discrete gradient operator as used in \cite{WangYangYinZhang:2008,ZhangBurgerOsher:2011}.
Correspondingly, the $x$-subproblem can be solved efficiently by the fast Fourier transform (FFT) and
the $y$-subproblem has an explicit solution given by the soft-thresholding \cite{WangYangYinZhang:2008}.
Therefore, Algorithm \ref{alg:noise} can be
efficiently implemented.

Figure \ref{fig:admmtv} reports numerical results by Algorithm \ref{alg:noise} on test
images Cameraman ($256\times256$) and Pirate ($512\times512$) with motion blur (\texttt{fspecial('motion',35,50)})
and Gaussian blur (\texttt{fspecial('gaussian',[20 20], 20)}) respectively. The noise level are $\delta = 0.256$
and $0.511$ respectively.  As comparisons, we also include the results obtained by FTVd v4.1 in
\cite{WangYangYinZhang:2008} which is a state-of-art algorithm for image deblurring. We can see that
the images reconstructed by our proposed ADMM have comparable
quality as the ones obtained by FTVd v4.1 with similar PSNR (peak sigal-to-noise ratio), while the choice of the regularization
parameter is not needed in our algorithm. Here the PSNR is defined by
$$
\mathrm{PSNR}=10\cdot \log_{10}\frac{255^2}{\mathrm{MSE}}[\mathrm{dB}],
$$
where MSE stands for the mean-squared-error per pixel.

\begin{figure}[ht!]
  \centering
  \begin{tabular}{cccc}
    \includegraphics[trim = 1cm 0cm 1cm 0cm, clip=true,width=0.22\textwidth]{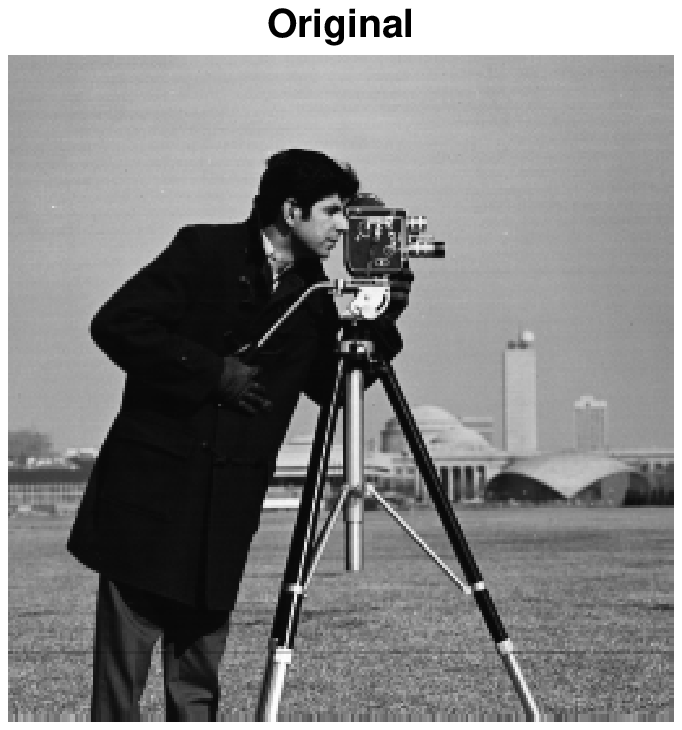}
    & \includegraphics[trim = 1cm 0cm 1cm 0cm, clip=true,width=0.22\textwidth]{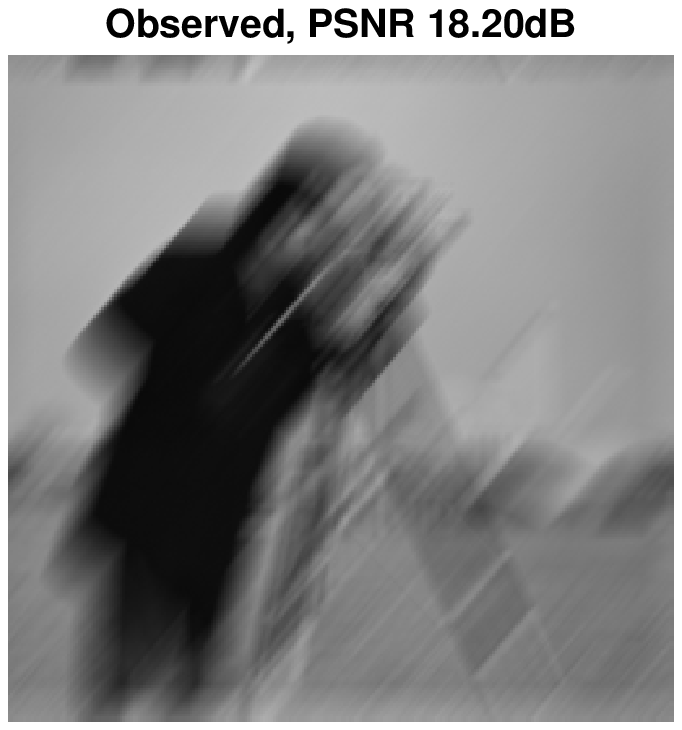}
    &\includegraphics[trim = 1cm 0cm 1cm 0cm, clip=true,width=0.22\textwidth]{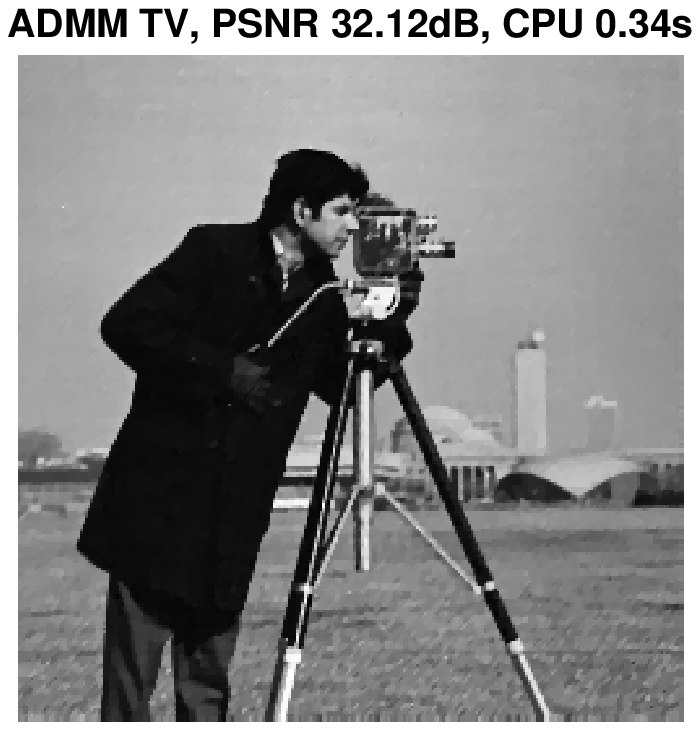}
    &\includegraphics[trim = 1cm 0cm 1cm 0cm, clip=true,width=0.22\textwidth]{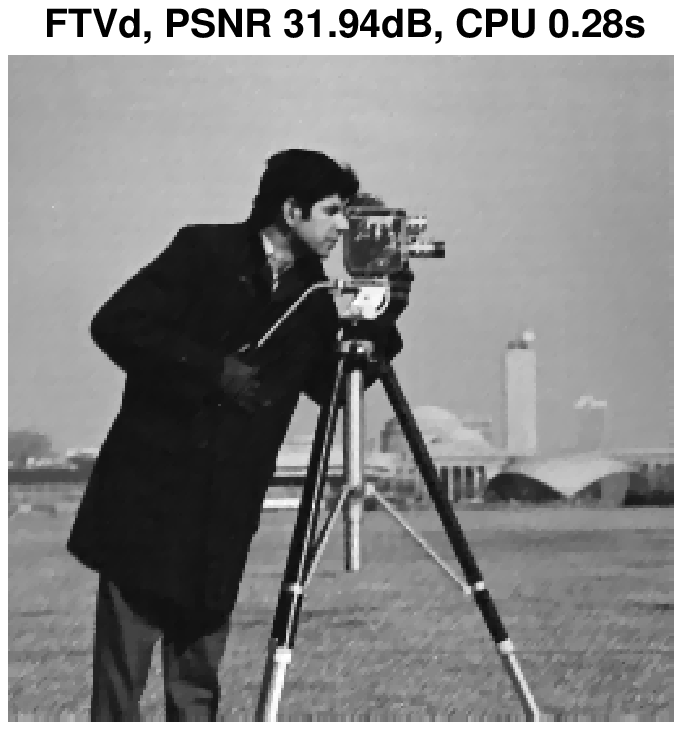}\\
        \includegraphics[trim = 1cm 0cm 1cm 0cm, clip=true,width=0.22\textwidth]{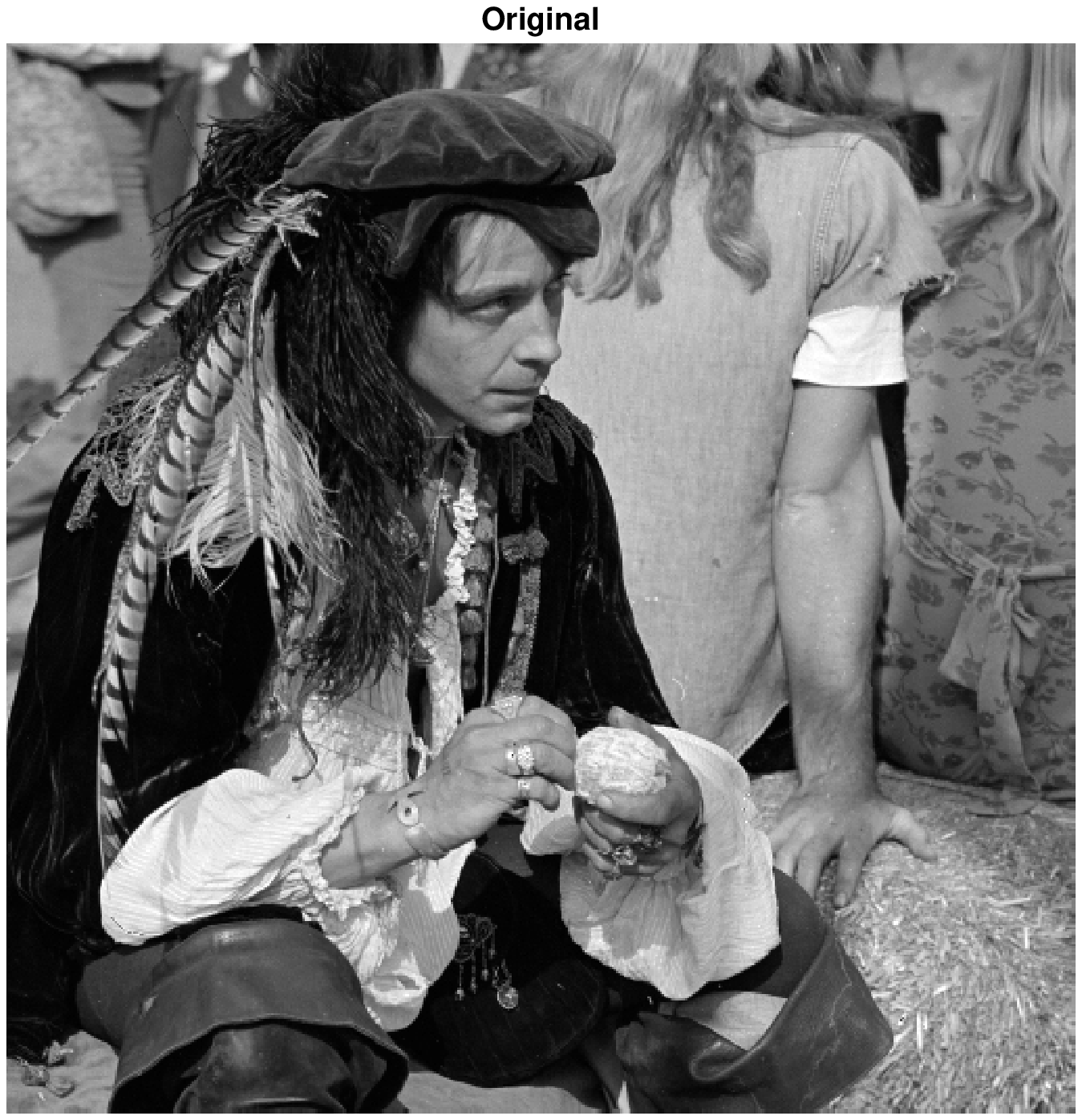}
    & \includegraphics[trim = 1cm 0cm 1cm 0cm, clip=true,width=0.22\textwidth]{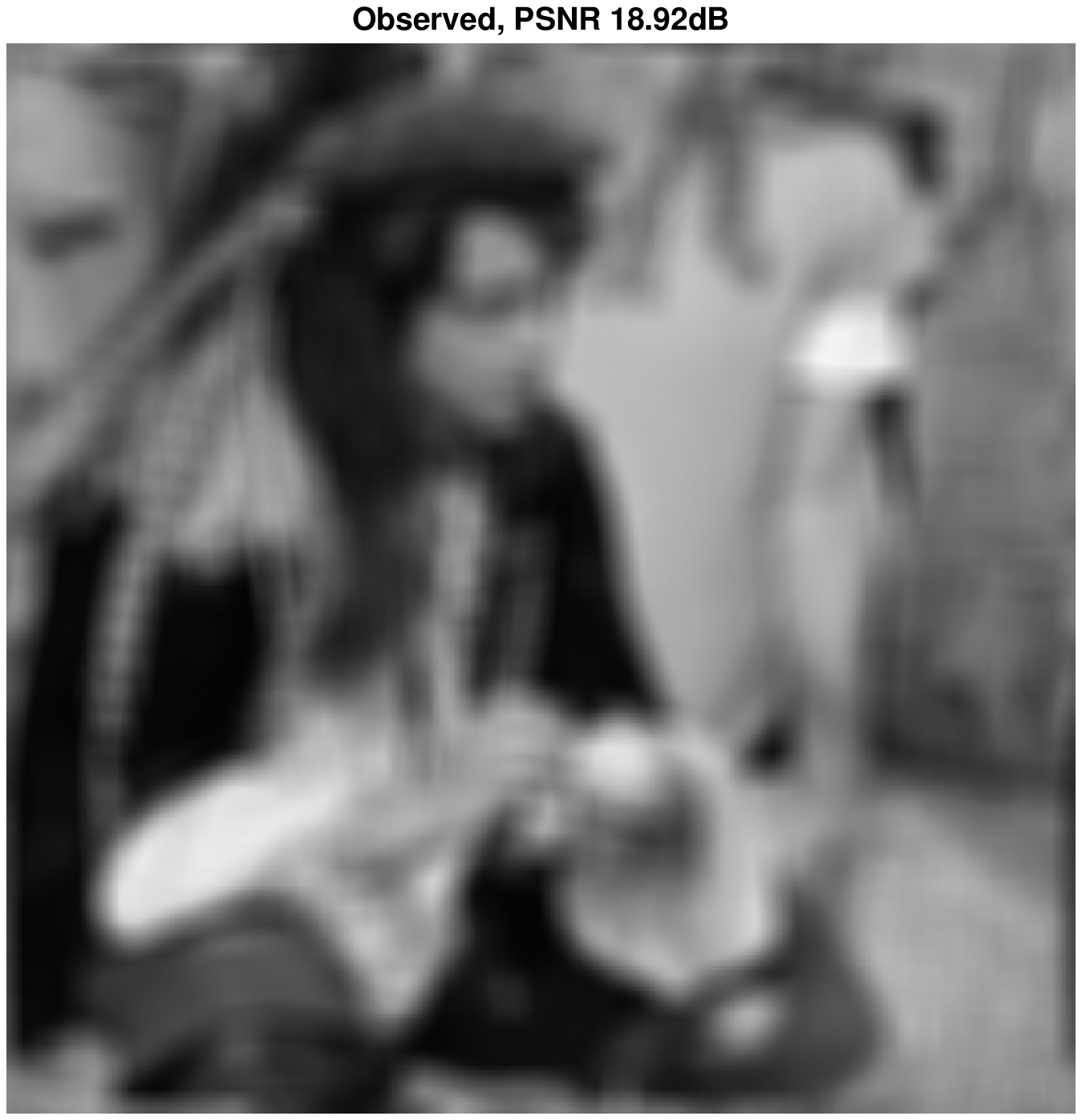}
    &\includegraphics[trim = 1cm 0cm 1cm 0cm, clip=true,width=0.22\textwidth]{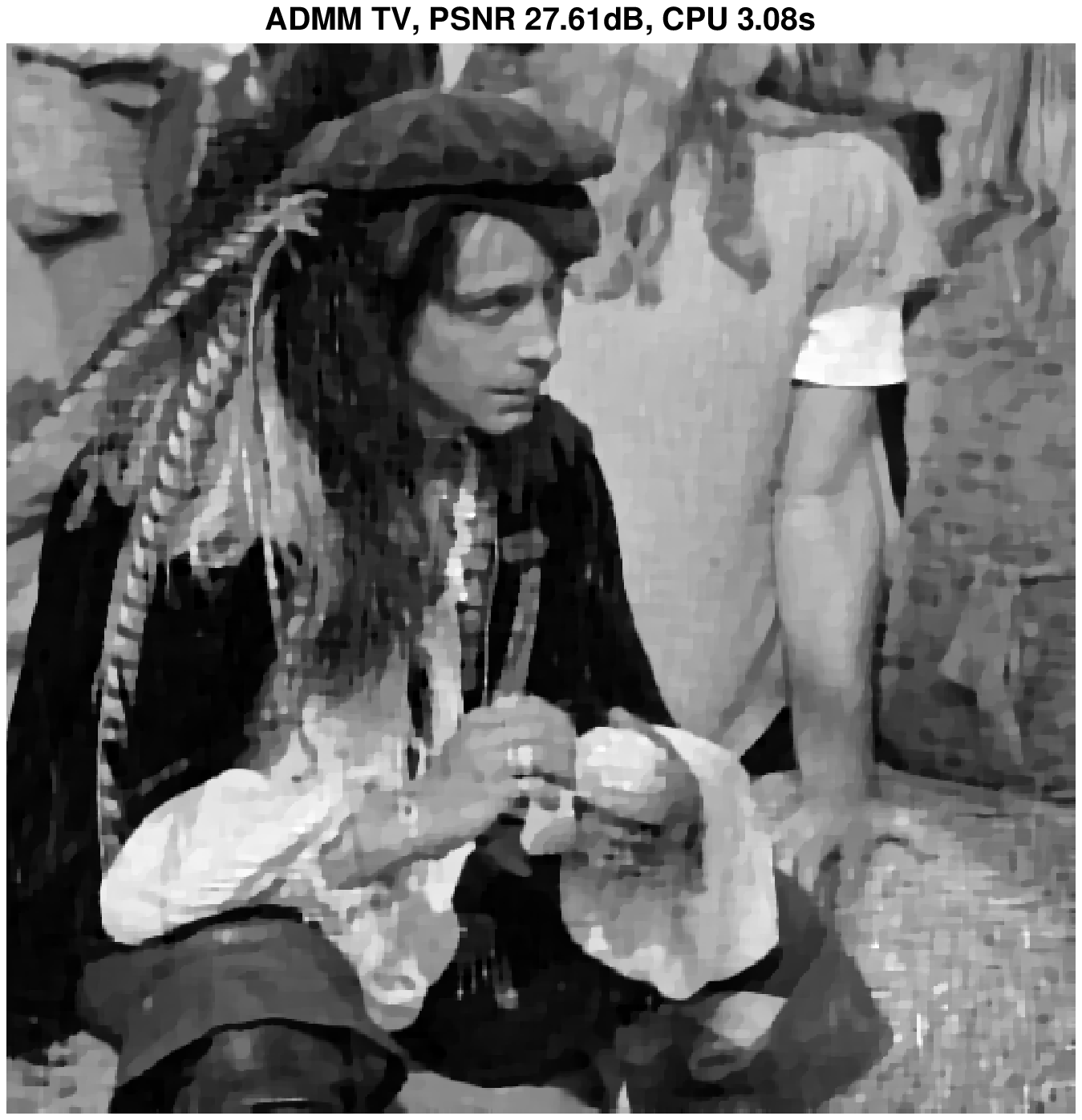}
    &\includegraphics[trim = 1cm 0cm 1cm 0cm, clip=true,width=0.22\textwidth]{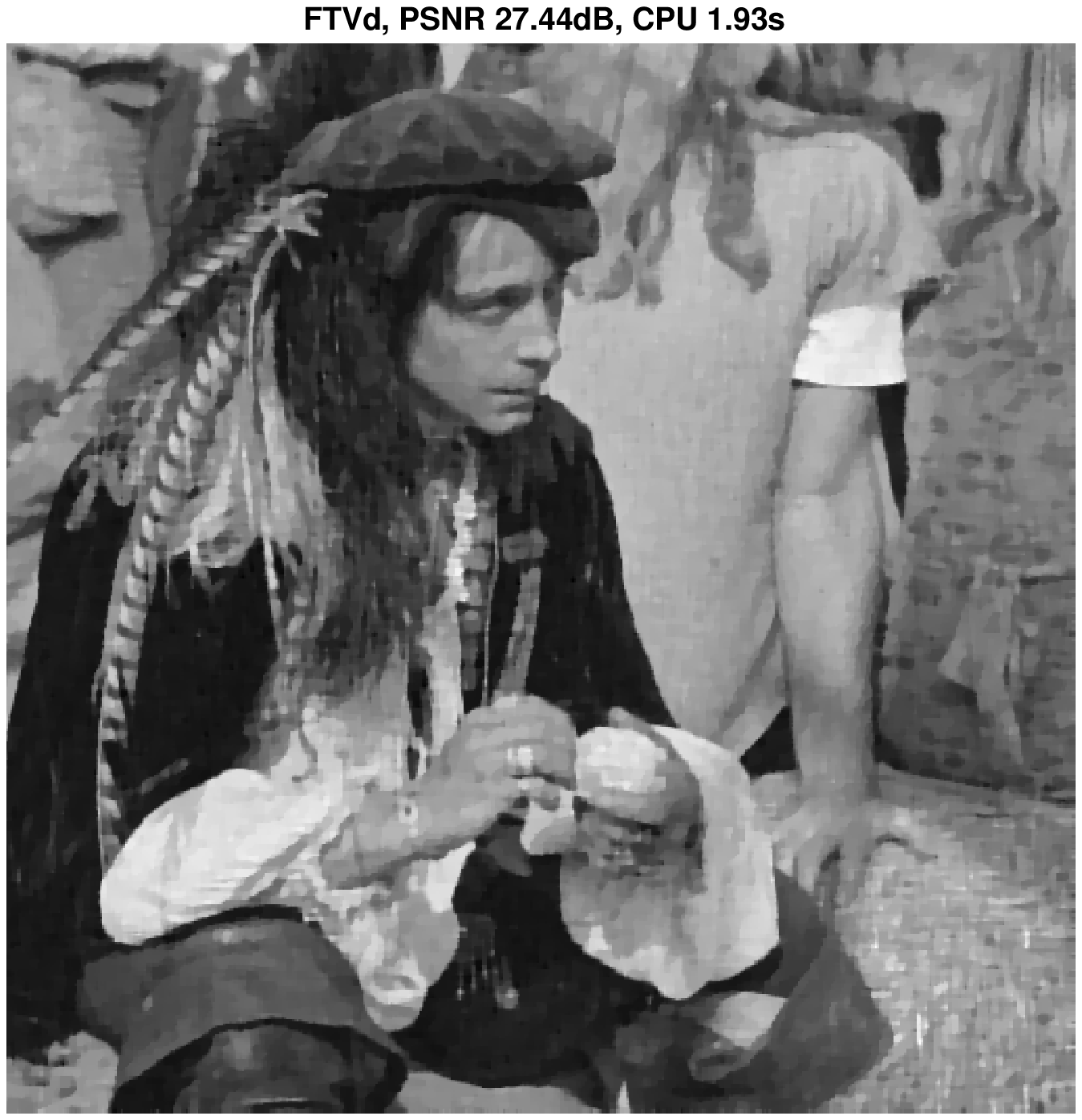}
    \end{tabular}
  \caption{ADMM-TV Debluring}\label{fig:admmtv}
\end{figure}

\subsection{ Two-dimensional framelet deblurring}

In this subsection we show the performance of Algorithm \ref{alg:noise} for image deblurring using
wavelet frames \cite{CaiDongOsherShen:2012,CaiDongShen:2014,DongJiangShen:2013,DongShen:2010,Shen:2010}.

In our numerical simulations on grayscale digital images represented by $M\times N$ arrays, we will use
the two-dimensional Haar with three level decomposition and piecewise linear B-spline framelets with one level decomposition, which can be constructed by taking tensor
products of univariate ones \cite{ChaiShen:2007}.
The action of the discrete framelet transform and its adjoint on images can be implemented implicitly by the MRA-based
algorithms \cite{DaubechiesHanRonShen:2003}. Assuming the periodic boundary condition on images,
the $x$-subproblem in Algorithm \ref{alg:noise} then can be solved by FFT. The $y$-subproblem
can be solved by the soft thresholding. Thus, Algorithm \ref{alg:noise} can be efficiently implemented.
Figure \ref{fig:admmframe} reports the reconstruction results using the test images Phantom ($256\times256$)
and Peppers ($256\times256$) with motion blur (\texttt{fspecial('motion',50,90)}) and
Gaussian blur (\texttt{fspecial('gaussian',[20 20], 30)}) respectively. The noise leve is $\delta=0.256$
for two examples. These results indicate the satisfactory performance of our proposed ADMM.

\begin{figure}[ht!]
  \centering
  \begin{tabular}{cccc}
    \includegraphics[trim = 0.5cm 0cm 0.5cm 0cm, clip=true,width=0.3\textwidth]{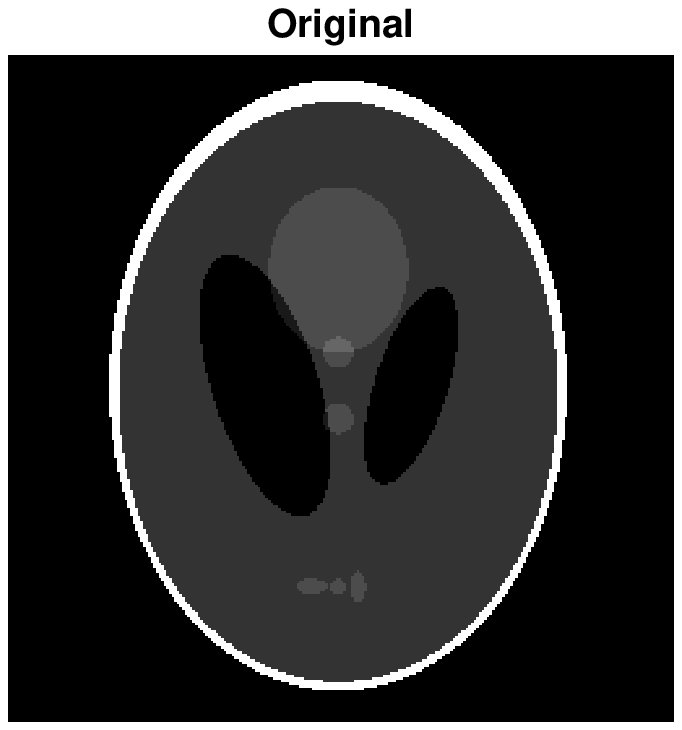}
    & \includegraphics[trim = 0.5cm 0cm 0.5cm 0cm, clip=true,width=0.3\textwidth]{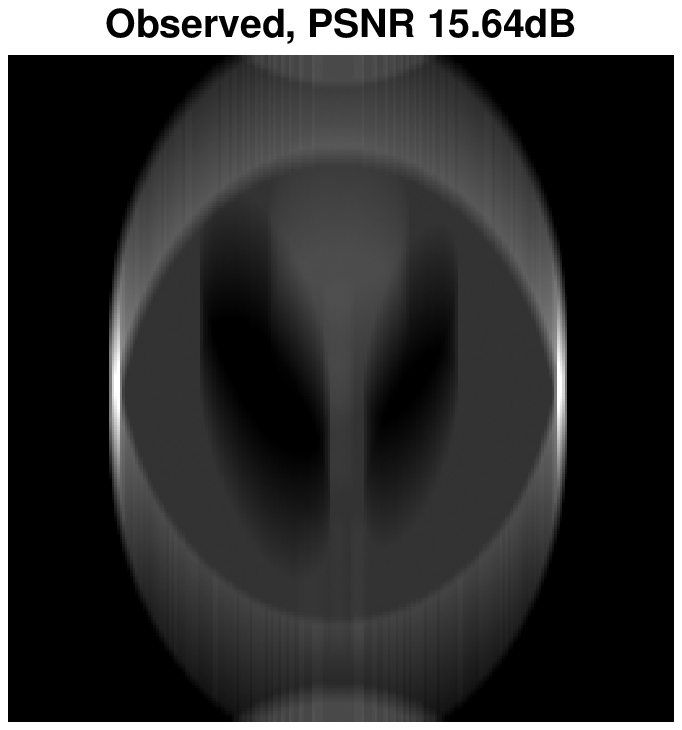}
    &\includegraphics[trim = 0.5cm 0cm 0.5cm 0cm, clip=true,width=0.3\textwidth]{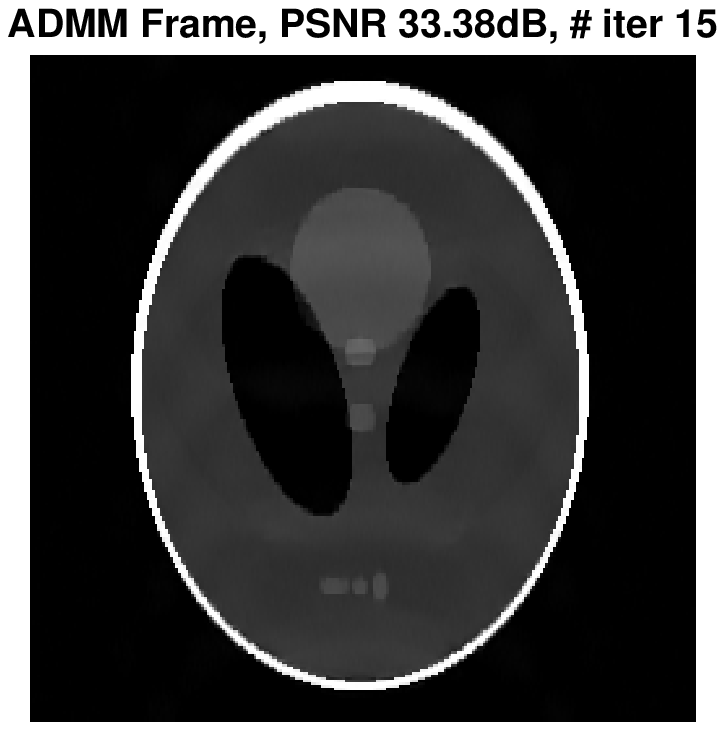}\\
     \includegraphics[trim = 0.5cm 0cm 0.5cm 0cm, clip=true,width=0.3\textwidth]{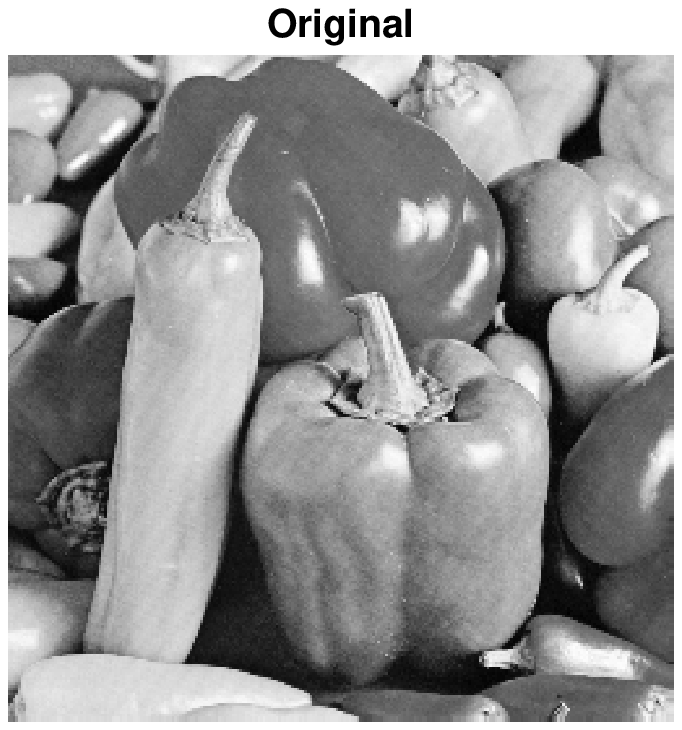}
    & \includegraphics[trim = 0.5cm 0cm 0.5cm 0cm, clip=true,width=0.3\textwidth]{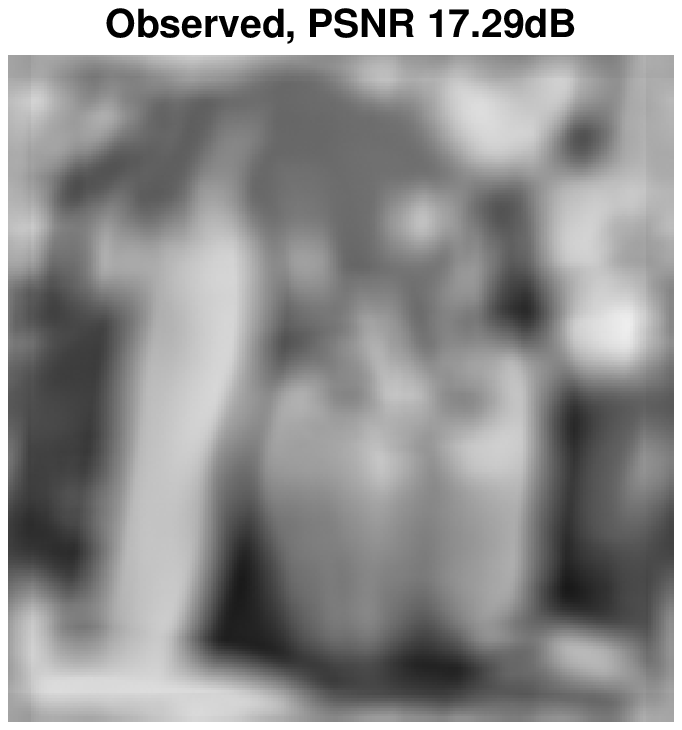}
    &\includegraphics[trim = 0.5cm 0cm 0.5cm 0cm, clip=true,width=0.3\textwidth]{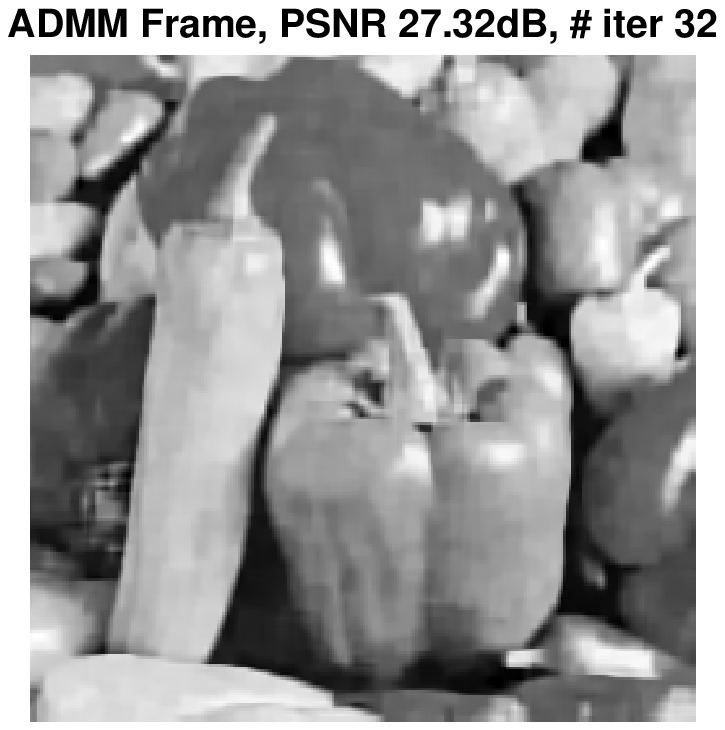}
  \end{tabular}
  \caption{ADMM-Framelet}\label{fig:admmframe}
\end{figure}

\subsection{Semi-convergence}\label{sec:3.4}

In this subsection we investigate the numerical performance of our ADMM if it is not terminated properly.
We use the 2-dimensional Cameraman TV deblurring as an example and run our ADMM until a preassigned maximum number of iterations (500) is
achieved. In Figures \ref{fig:semicon2dtv} we plot the corresponding results on
the PSNR values and the values of $E_k^\d$ versus the number of iterations. It turns out that
the PSNR increases first and then decreases after a critical number of iterations. This illustrates the
semi-convergence property of our proposed ADMM in the framework of iterative regularization methods
and indicates the importance of terminating the iteration by a suitable stop rule.


  \begin{figure}[ht!]
  \centering
  \begin{tabular}{cccc}
    \includegraphics[trim = 0.8cm 0cm 0.5cm 0cm, clip=true,width=0.22\textwidth]{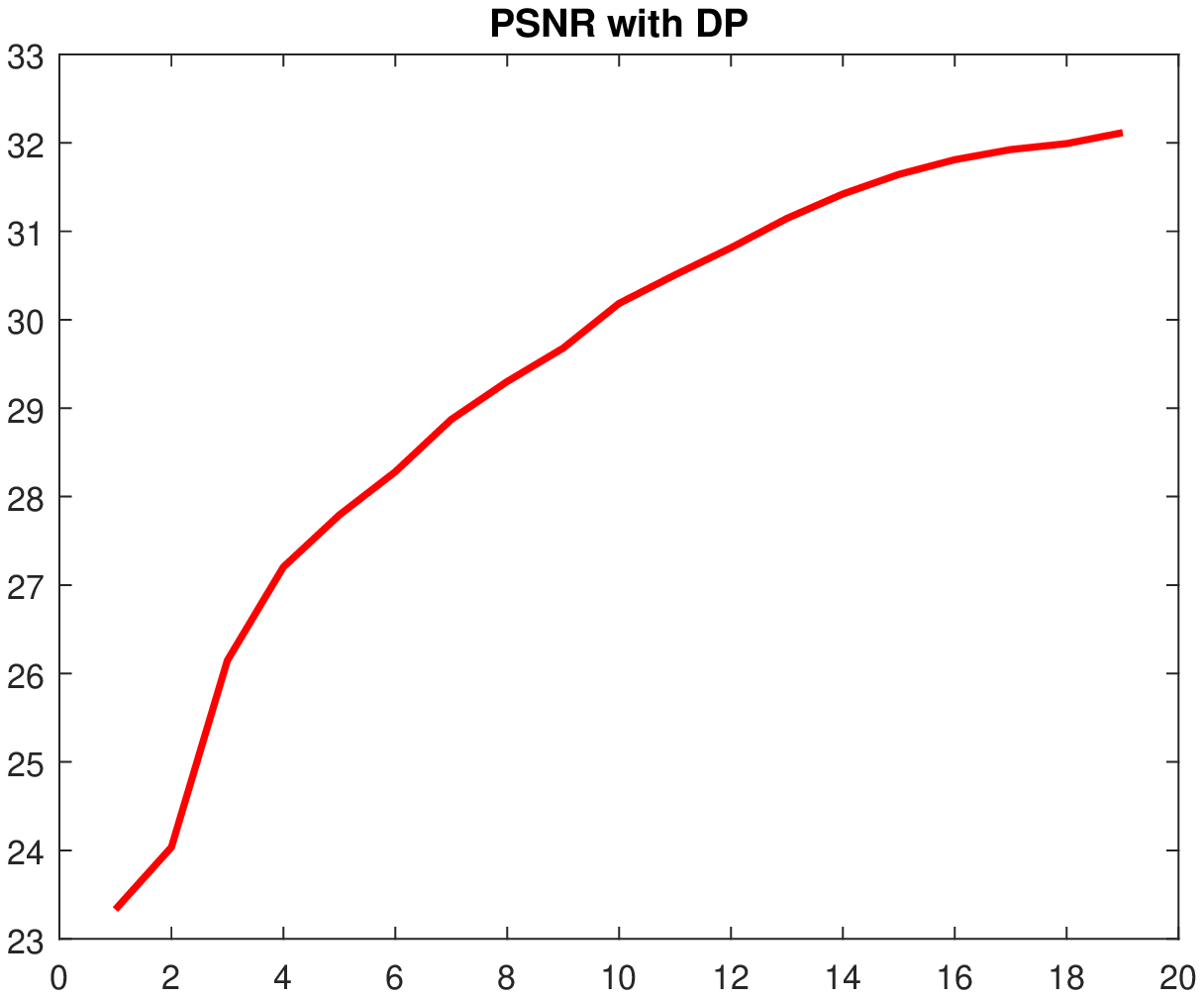}
    & \includegraphics[trim = 0.8cm 0cm 0.5cm 0cm, clip=true,width=0.22\textwidth]{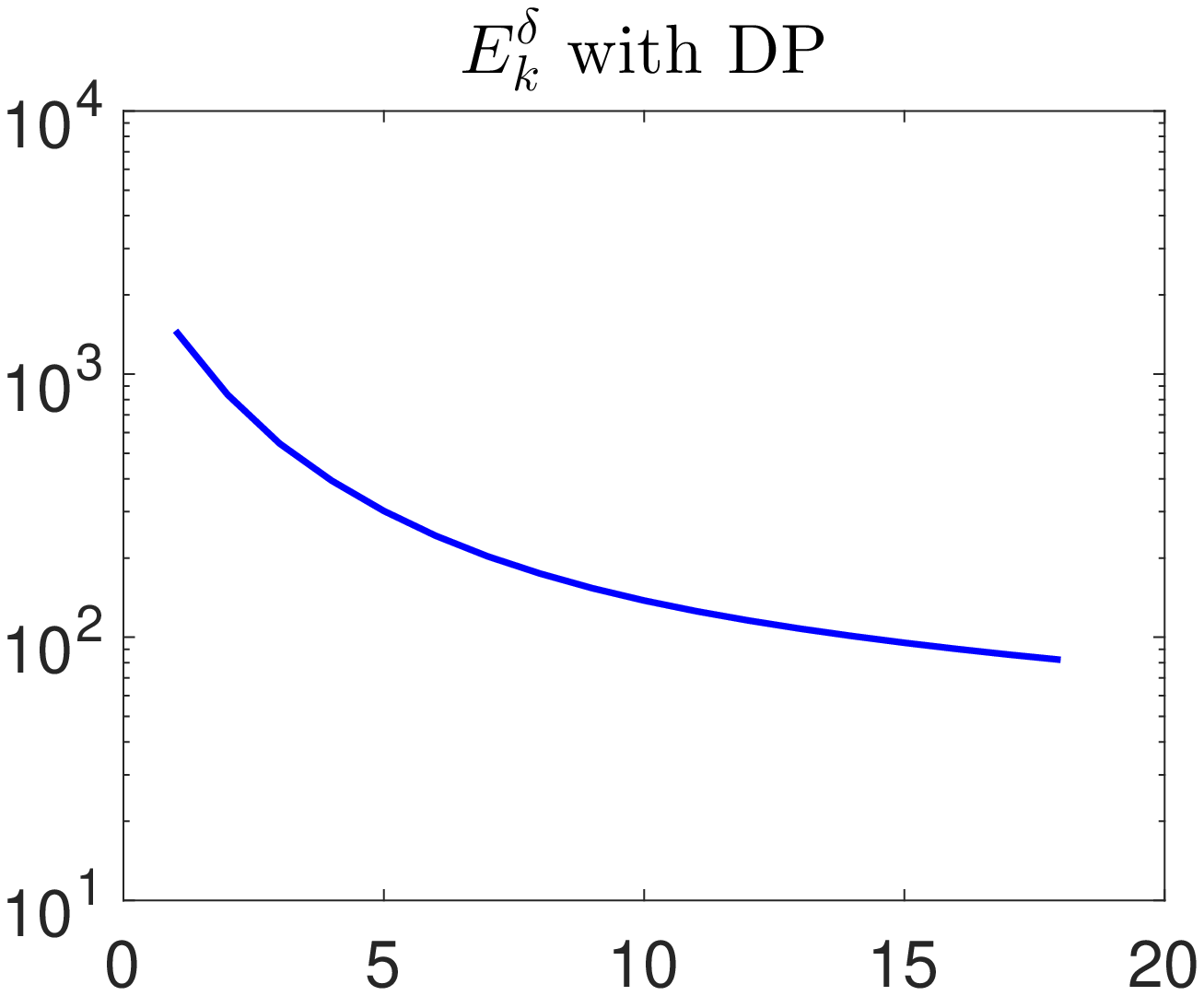}
    &\includegraphics[trim = 0.8cm 0cm 0.5cm 0cm, clip=true,width=0.22\textwidth]{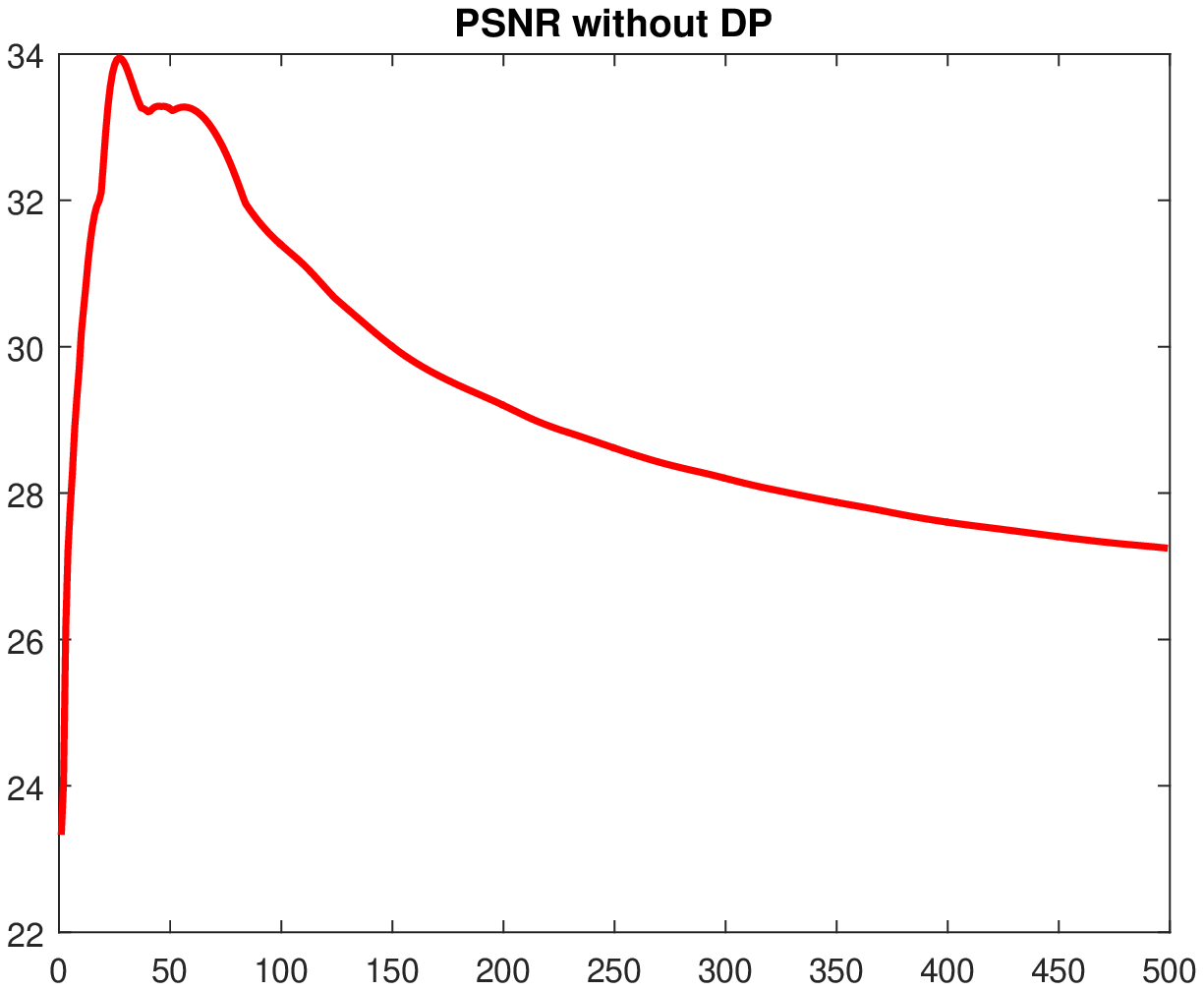}
    &\includegraphics[trim = 0.8cm 0cm 0.5cm 0cm, clip=true,width=0.22\textwidth]{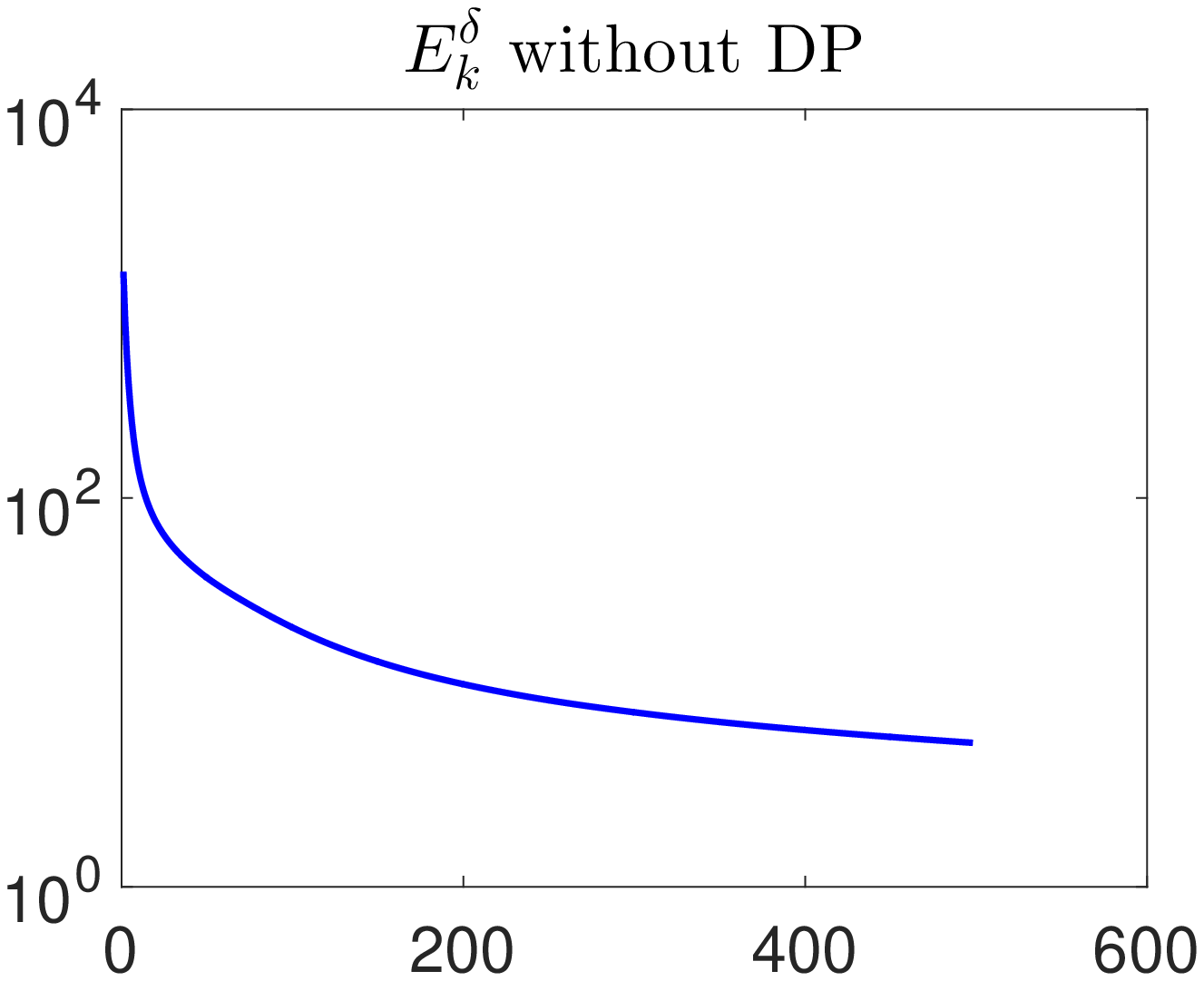}\\
    \multicolumn{2}{c}{Stop by \eqref{eq:stop} } & \multicolumn{2}{c}{Nonstop}
  \end{tabular}
  \caption{Semi-convergence test}\label{fig:semicon2dtv}
\end{figure}


\subsection{Sensitivity on $\rho_1$ and $\rho_2$}

In this subsection we illustrate that the performance of Algorithm \ref{alg:noise} is not quite sensitive to
the two parameters $\rho_1$ and $\rho_2$. To this end we run our ADMM on the Cameraman TV deblurring  problem  with same parameters
as in Section 3.2 with different $\rho_1$ and $\rho_2$. The PSNR and the number of  iterations are given in
Table \ref{tab:1} from which we can see that the reconstructions remain stable for a large range of values of
$\rho_1$ and $\rho_2$.

\begin{table}[h]
 \caption{PSNR and number of iteration for different $\rho_1$ and $\rho_2$.}\label{tab:1}
  \begin{center}
 \begin{tabular}{c|ccccc}
 \hline
   \backslashbox{$\rho_2$}{$\rho_1$}  & 250& 500 & 1000 & 2000& 4000      \\
 \hline
   2.5 &   (31.9,   61  )          &   (32.1, 30   )    &   ( 32.7, 15  )    &   (31.7,  9  )      &   (30.6,  5 )\\
   5   &   (32.0,   64  )          &   (32.9, 33   )    &   ( 32.4, 17  )    &   (31.6, 10  )      &   (31.5,  6 )\\
   10  &   (32.5,   68  )          &   (33.0, 36   )    &   ( 32.1, 20  )    &   (32.2, 11  )      &   (32.0,  7 )\\
   20  &   (33.1,   74  )          &   (32.3, 40   )    &   ( 32.5, 23  )    &   (32.7, 14  )      &   (32.4,  9 )\\
   40  &   (32.8,   82  )          &   (32.6, 47   )    &   ( 32.8, 28  )    &   (32.5, 17  )      &   (32.0, 11 )\\
   80  &   (32.7,   97  )          &   (32.9, 58   )    &   ( 32.8, 36  )    &   (32.3, 23  )      &   (31.6, 15 )\\
   160 &   (32.9,  120  )          &   (32.9, 76   )    &   ( 32.5, 50  )    &   (31.9, 34  )      &   (31.3, 24 )\\
  \hline
  \end{tabular}
 \end{center}
\end{table}


\section{Conclusion}
In this work we propose an alternating direction method of multiplies to solve inverse problems.
When the data is given exactly, we prove the convergence of the algorithm without using the existence of Lagrange
multipliers. When the date contains noise, we propose a stop rule and show that our ADMM renders into a
regularization method. Numerical simulations are given to show the efficiency of the proposed algorithm.

There are several possible extensions for this work. First, in our ADMM for solving inverse problems, we used
two parameters $\rho_1$ and $\rho_2$ which are fixed during iterations. It is natural to consider the situation
that $\rho_1$ and $\rho_2$ change dynamically. Variable step sizes have been used in the augmented Lagrangian method
to reduce the number of iterations (see \cite{FrickScherzer:2010,FrickLorenzElena:2011,JinZhong:2014}).
It would be interesting to investigate what will happen if dynamically changing parameters $\rho_1$ and $\rho_2$ are used in our ADMM.
Second, the $x$-subproblem in our ADMM requires to solve linear systems related to $\rho_1 A^*A + \rho_2 W^*W$.
In general, solving such linear system is very expensive. It might be possible to remedy this drawback by applying the linearization and/or precondition strategies. Finally, in applications where the sought solution is a priori known to satisfy certain constraints,
it is of interest to consider how to incorporate such constraints into our ADMM in an easily implementable way
and to prove some convergence results.

\section*{Acknowledgment}
Y. Jiao is partially supported by National Natural Science Foundation of China No. 11501579,
Q. Jin is partially supported by the discovery project grant DP150102345 of Australian Research Council
and X. Lu is partially supported by the National Natural Science Foundation of China No. 11471253.

\bibliographystyle{abbrv}
\bibliography{admminv}
\end{document}